\documentclass[12pt]{article}

\usepackage[left=2.5cm, right=2.5cm, top=2.5cm, bottom=2.5cm]{geometry}
\usepackage[utf8]{inputenc}
\usepackage[english]{babel}
\usepackage{amsthm}
\usepackage{amsmath}
\usepackage{amsfonts}
\usepackage{appendix}
\usepackage{enumerate}
\usepackage[normalem]{ulem}
\usepackage{mathtools}
\usepackage{stmaryrd}
\usepackage{cases}
\usepackage{xcolor}
\usepackage[noadjust]{cite}
\usepackage[shortlabels]{enumitem}
\usepackage[hidelinks]{hyperref}

\DeclareMathOperator{\rank}{rank}
\DeclareMathOperator{\diag}{diag}
\DeclareMathOperator{\Ran}{Ran}

\DeclarePairedDelimiter{\abs}{\lvert}{\rvert}
\DeclarePairedDelimiter{\norm}{\lVert}{\rVert}
\DeclarePairedDelimiter{\scalprod}{\langle}{\rangle}

\newcommand{\suchthat}{\ifnum\currentgrouptype=16 \mathrel{}\middle|\mathrel{}\else\mid\fi}

\newtheorem{theorem}{Theorem}[section]
\newtheorem{lemma}[theorem]{Lemma}
\newtheorem{proposition}[theorem]{Proposition}
\theoremstyle{definition}
\newtheorem{definition}[theorem]{Definition}
\newtheorem{remark}[theorem]{Remark}
\newtheorem*{hypothesis}{Hypothesis}

\begin{document}

\title{Exact controllability and rapid stabilization of linear difference delay equations and one-dimensional hyperbolic systems}

\makeatletter
{\def\@thefnmark{}\def\hyper@@anchor{}\@footnotetext{The first and third authors were funded by the French Agence Nationale de la Recherche (Grant ANR-26-CE40-0732-01). The third author was funded by the French Agence Nationale de la Recherche (Grant ANR-22-CPJ2-0138-01). The fourth author was partially supported by the French government under the France 2030 program, reference ANR-11-IDEX-0003 within the OI H-Code.}}
\makeatother

\author{
Yacine Chitour\thanks{Universit\'e Paris-Saclay, CNRS, CentraleSup\'elec, Laboratoire des signaux et syst\`emes, 91190 Gif-sur-Yvette, France (\url{yacine.chitour@centralesupelec.fr}).}
\and
S\'ebastien Fueyo\thanks{IMS Lab, Univ. Bordeaux, Bordeaux INP, CNRS (UMR 5218), 33405 Talence, France (\url{sebastien.fueyo@u-bordeaux.fr}).}
\and
Pierre Lissy\thanks{CERMICS, CNRS, ENPC, Institut Polytechnique de Paris, Marne-la-Vall\'ee, France (\url{pierre.lissy@enpc.fr}).}
\and
Guilherme Mazanti\thanks{Universit\'e Paris-Saclay, CNRS, CentraleSup\'elec, Inria, Laboratoire des signaux et syst\`emes, 91190 Gif-sur-Yvette, France (\url{guilherme.mazanti@inria.fr}).}~\thanks{Fédération de Mathématiques de CentraleSupélec, 91190, Gif-sur-Yvette, France.}
}

\date{}

\maketitle

\begin{abstract}
This paper addresses exact controllability and rapid stabilization issues for linear difference delay systems with finitely many pointwise delays. We first establish a criterion for exact controllability in $L^r$, $r \in [1, +\infty)$, characterized by a necessary and sufficient Hautus-type frequency-domain condition. The sufficiency of this condition in \(L^r\) was already known \cite{fueyo:hal-04830757}, whereas its necessity had previously been established only in \(L^1\) \cite{chitour:hal-03827918}. We answer \cite[Conjecture~3.8]{chitour:hal-03827918} by proving that this condition is also necessary for $r \in (1, +\infty)$. The analysis is based on rewriting the system as an evolution equation in a suitable Banach space, which allows one to apply duality arguments in the spirit of \cite{Russell_Weiss94}. Building on this controllability analysis and the abstract stabilization framework developed in \cite{Komornik, Vest2013Rapid}, we prove that our exact controllability criterion in \(L^r\) is also a sufficient condition for rapid stabilization in $L^2$, \textit{i.e.}, that the system of interest can be stabilized by a family of state feedback control laws ensuring that the closed-loop system is exponentially stable in \(L^2\) with an arbitrarily large decay rate. As an application, we extend our results to linear one-dimensional hyperbolic partial differential equations, which can be reduced to linear difference delay systems via the method of characteristics, obtaining also a necessary and sufficient frequency-domain criterion for exact controllability in $L^r$, $r\in [1, +\infty)$, which, as before, is a sufficient condition for rapid stabilization in $L^2$.
\end{abstract}

\bigskip


\medskip


\tableofcontents

\section{Introduction}

The present paper deals with the exact controllability and stabilization of linear difference delay equations of the form
\begin{equation}
\label{main_system}
\left\{
\begin{aligned}
x(t) & = \sum_{j=1}^N A_j x(t-\Lambda_j) + Bu(t), & \quad & t \ge 0, \\
x(\theta) & = x_0(\theta), & & \theta \in [-\Lambda_N,0),
\end{aligned}
\right.
\end{equation}
where, given positive integers $d$, $m$, and $N$, one has that $A_1,\dotsc,A_N$ are fixed $d\times d$ matrices with complex entries\footnote{Even though we work in this paper with complex-valued matrices and vectors, the main results of the paper remain true in the real-valued setting (see Remark~\ref{remk:Cd-Rd}).}, $B$ is a $d\times m$ matrix with complex entries, the state $x(t)$ and the control $u(t)$ at a given time \(t\) belong to $\mathbb{C}^d$ and $\mathbb{C}^m$, respectively, the delays $\Lambda_1, \dots, \Lambda_N$ are positive real numbers satisfying $\Lambda_1< \dots<\Lambda_N$, and the initial condition \(x_0\) is a given function defined on \([-\Lambda_N, 0)\) and taking values in \(\mathbb C^d\).

System~\eqref{main_system} can be seen as an infinite-dimensional control system on the state space $\mathsf X^r=L^r((-\Lambda_N, 0),\allowbreak \mathbb{C}^d)$ with control functions in $\mathsf U^r=L^r_{\rm loc}(\mathbb{R}_+, \mathbb{C}^m)$, where $r\in [1, +\infty)$. Indeed, given \(x_0 \in \mathsf X^r\) and \(u \in \mathsf U^r\), a solution \(x\) of \eqref{main_system} can be seen as a curve $t \mapsto x_t \in \mathsf X^r$, where \(x_t\) is the \emph{history function} associated with \(x\), defined for \(t \geq 0\) by $x_t(\theta) = x(t + \theta)$ for $\theta\in (-\Lambda_N, 0)$.

One of the aims of this paper is to study the \emph{$L^r$-exact controllability} problem for \eqref{main_system}, which consists in determining whether, for every $x_{\text{ini}}, x_{\text{fin}} \in \mathsf X^r$, there exists \(T > 0\) and a control function $u \in L^r((0, T), \mathbb{C}^m)$ such that the trajectory of System~\eqref{main_system} with initial condition $x_{\text{ini}}$ at $t=0$ and control \(u\) reaches $x_{\text{fin}}$ at time $T$. When the time $T$ does not depend on $x_{\text{ini}}, x_{\text{fin}}$, one speaks about $L^r$-exact controllability \emph{in time $T$} (see Definition~\ref{def:controllability} below).

\subsection{Hautus-type condition for \texorpdfstring{$L^r$}{Lr}-exact controllability}

Establishing Hautus-type frequency-domain criteria for controllability has played a central role in control theory since the seminal work \cite{hautus1969controllability} for ordinary differential equations. For delay systems such as System~\eqref{main_system}, an appropriate framework for studying controllability is provided in  \cite{yamamoto1989reachability}, where controllability criteria are established  for a broad class of systems, called \textit{pseudo-rational systems}, that naturally includes delay systems.

On the one hand, regarding approximate controllability, that is, the property that every target can be approached arbitrarily closely by suitable controls, the Hautus-type conditions derived in \cite{yamamoto1989reachability} apply to pseudo-rational systems and different classes of delay systems, and explicit \(L^r\)-approximate controllability criteria for \eqref{main_system} using these conditions were derived in \cite{chitour:hal-03827918}.
These results were later extended to exact controllability with targets in distributional spaces for certain classes of delay systems in \cite{Yamamoto_Willems}, but not for System~\eqref{main_system}. On the other hand, the recent works \cite{chitour:hal-03827918,Corona_fueyo,fueyo:hal-04830757} not only provided upper bounds on the minimal controllability time but also established necessary and sufficient Hautus-type criteria for the exact controllability of System~\eqref{main_system} in the $L^1$ setting, while only sufficient conditions were obtained in the $L^r$ framework, $r \in (1, +\infty)$, leaving open whether these conditions are also necessary. 

In the present paper,  we first prove that the Hautus-type conditions are also necessary in the $L^r$ setting (see Theorem~\ref{th:iff_controllability}). The proof relies on recasting System~\eqref{main_system} into an abstract control framework
in the spirit of~\cite{Tucsnak_Weiss}, and then applying the necessary
Hautus-type criterion for exact observability established in~\cite{Russell_Weiss94}, together with the standard
controllability--observability duality. Moreover, the rank conditions appearing in Theorem~\ref{th:iff_controllability} are independent of the exponent $r\in [1, +\infty)$. As a consequence, controllability in one space $L^r$ implies controllability in every other space $L^{\widetilde r}$, with $\widetilde r\in [1, +\infty)$, showing that controllability is in fact equivalent throughout the whole scale of $L^r$ spaces with finite \(r\).

\subsection{Rapid stabilization}

The stabilization problem of \eqref{main_system} (with a linear feedback) consists in designing a control $u$, depending linearly on the state $x$, such that the origin of the corresponding closed-loop system obtained from System~\eqref{main_system} is exponentially stable. One refers to rapid stabilization when one requires an arbitrarily large rate of exponential convergence in the stabilization problem.

Over the years, the stabilization of \eqref{main_system} has attracted considerable attention, although it has often been studied as an intermediate step toward the analysis of neutral delay differential systems. One may notably quote \cite{salamon1984control}, which established a necessary Hautus-type condition for stabilization in the general case, and additionally proved a sufficient condition when the delays are commensurate. In the latter case, the stabilization problem reduces to that of a single-delay difference equation and can be addressed through algebraic techniques. Another important contribution can be found in \cite{hale2002stabilization}, which derived a stabilization criterion for \eqref{main_system} valid for arbitrary sets of delays. However, their approach is non-constructive and therefore cannot be used to obtain a stabilization result for a prescribed set of delays.

The present paper establishes a sufficient spectral criterion for the rapid stabilization of \eqref{main_system}. Recall that, for linear time-invariant finite-dimensional control systems, controllability implies rapid stabilization by linear feedback laws. We  extend this result to \eqref{main_system}  by proving that our \(L^r\) exact controllability criterion also implies \(L^2\) rapid stabilization. The argument combines the abstract differential framework given in Section~\ref{sec:representation} and the rapid stabilization result in Hilbert spaces proved in \cite{Komornik, Vest2013Rapid}.

\subsection{One-dimensional hyperbolic systems}

Difference delay systems naturally arise in the modeling of physical phenomena involving transport effects, such as high-frequency electronic amplifiers \cite{Suarez}. In many cases, these phenomena are described by one-dimensional hyperbolic partial differential equations (PDEs). In this setting, the method of characteristics allows one to reduce the problem to the study of a subclass of difference delay systems of the form~\eqref{main_system}, a classical technique widely used in the literature to study well-posedness, qualitative behavior, stability, controllability, and stabilizability properties of hyperbolic PDEs \cite{auriol2019explicit, baratchart, bastin2016stability, Brayton1966Bifurcation, Chitour2020Approximate, Chitour2016Stability, Cooke1968Differential, CoNg, Greenberg1984Effect, Slemrod1971Nonexistence, coron}. Results obtained for delay systems are therefore widely used in the analysis of one-dimensional hyperbolic PDEs.

The analysis of controllability notions for one-dimensional first-order hyperbolic systems dates back at least to the seminal paper~\cite{Russell1978Controllability} and has been extensively studied since then, with many works, such as \cite{CoronOptimal, Coron2019Optimal, Coron2021Null, Gantouh2020ApproximateCO, Engel2008Vertex, CHITOUR2024_contr_hyp, Hu2024Minimal, Hu2025Method, Weck1982Remark}, dealing with characterizations of several forms of controllability (exact, approximate, or null controllability) and the search for a minimal time for attaining such control objectives. We refer the interested reader to the recent survey~\cite{HuOliveSurvey}.

Stabilization issues for linear one-dimensional hyperbolic systems have also been the subject of several works with typically two complementary approaches. The first one, based on the backstepping method, consists in designing a feedback law after transforming the system into a simpler and more tractable target system \cite{vazquez2011backstepping, krstic2008backstepping, coron2020finite, coron2017finite}. The second approach relies on the equivalence between delay systems and hyperbolic PDEs, and uses frequency-domain stability results for System~\eqref{main_system} to derive stabilization properties for certain PDEs \cite{bastin2016stability, auriol2019explicit}, sometimes in combination with backstepping techniques. 

From this perspective, we aim to exploit the $L^r$ exact controllability result and the $L^2$ rapid stabilization criterion established for System~\eqref{main_system} for these hyperbolic systems, which are modeled as a network of pure transport equations on segments of length one
where the coupling and the actuation occur linearly and at the nodes. Relying on the method of characteristics, it is classical to show that such a system of hyperbolic equations can be cast as a particular instance of System~\eqref{main_system}. However, one cannot expect controllability in $\mathsf X^r$ spaces, but instead in a subspace $\mathsf S^r$, $r\in [1, +\infty]$, presented in Section~\ref{sec:PDE}, as remarked in \cite{CHITOUR2024_contr_hyp}. 
Controllability in \(\mathsf X^r\) is a stronger requirement and leads to more restrictive rank conditions. We must therefore adapt our framework to \(\mathsf S^r\), as detailed in Section~\ref{sec:PDE}. Since the strategy of proofs is identical to the one developed for the $\mathsf X^r$ spaces, we only provide the final computations with no further details. That yields a necessary and sufficient frequency-domain controllability condition for $L^r$ exact controllability in finite time, together with an upper bound on the minimal controllability time, as well as a sufficient spectral criterion for \(L^2\) rapid stabilization in this framework.

\subsection{Plan of the paper}
In Section~\ref{sec:representation}, we recall some previous results on System~\eqref{main_system} before reformulating it as an abstract evolution equation in a Banach space, using a framework in the spirit of \cite{Tucsnak_Weiss}. Section~\ref{sec:controllability} is devoted to exact controllability and presents our main result on System~\eqref{main_system}, Theorem~\ref{th:iff_controllability}, with a necessary and sufficient frequency-domain criterion for exact controllability. In Section~\ref{sec:stabilization}, a sufficient spectral criterion for rapid stabilization of the delay system is proved, using the abstract formulation and following the strategy devised in \cite{Komornik, Vest2013Rapid}. Finally, Section~\ref{sec:PDE} is devoted to adapting the previous results for one-dimensional hyperbolic systems. 

\subsection{Notation}

In this paper, we consider that the vector space \(\mathbb C^d\) for some positive integer \(d\) is endowed with its standard Hermitian inner product $\langle \cdot,\cdot \rangle_{\mathbb{C}^d}$ and we denote by $\|\cdot\|_{\mathbb{C}^d}$ the associated Hermitian norm. Given two positive integers \(d, m\), the set of \(d \times m\) matrices with complex coefficients is denoted by \(\mathbb C^{d \times m}\), and this space is endowed with the operator norm \(\norm{\cdot}_{\mathbb C^{d \times m}}\) induced by the Hermitian norms \(\norm{\cdot}_{\mathbb C^m}\) and \(\norm{\cdot}_{\mathbb C^d}\). For a given matrix \(M \in \mathbb C^{d \times m}\), we denote by \(M^T\) its transpose matrix and by \(M^\ast\) its conjugate transpose matrix. Whenever needed, vectors in $\mathbb{C}^d$ are regarded as column vectors, \emph{i.e.}, we identify \(\mathbb C^d\) with \(\mathbb C^{d \times 1}\). Given positive integers \(d, m_1, m_2\) and two matrices \(M \in \mathbb C^{d \times m_1}\) and \(B \in \mathbb C^{d \times m_2}\), we denote by \([M, B] \in \mathbb C^{d \times (m_1 + m_2)}\) the matrix obtained by placing side by side the columns of \(M\) and \(B\).

The norm in a Banach space \(\mathsf X\) is denoted by \(\norm{\cdot}_{\mathsf X}\) and, when \(\mathsf X\) is a space of the form \(L^r((a, b), \mathbb C^d)\) for some \(r \in [1, +\infty]\), \(a, b \in \mathbb R\), and positive integer \(d\) and the values of \(a, b, d\) are clear from the context, we denote the norm simply by \(\norm{\cdot}_{L^r}\). The identity operator in a Banach space \(\mathsf X\) is denoted by \(I_{\mathsf X}\), or simply by \(I\) when \(\mathsf X\) is clear from the context. When \(\mathsf X = \mathbb C^d\), we denote \(I_{\mathbb C^d}\) simply by \(I_d\), and we identify it with the identity matrix whenever needed. Given two Banach spaces \(\mathsf X\) and \(\mathsf Y\), we denote by \(\mathcal L(\mathsf X, \mathsf Y)\) the Banach space of bounded linear operators from \(\mathsf X\) to \(\mathsf Y\), endowed with its standard operator norm. For simplicity, we set \(\mathcal L(\mathsf X) = \mathcal L(\mathsf X, \mathsf X)\).

We shall work in this paper with \(C^0\)-semigroups in Banach spaces, and we refer to the standard references \cite{Engel2001, Pazy1983Semigroups} for the classical theory of \(C^0\)-semigroups. We shall say that a semigroup \((\Upsilon(t))_{t \geq 0}\) in a Banach space \(\mathsf X\) is exponentially stable if there exist constants \(C \geq 1\) and \(\gamma > 0\) such that \(\norm{\Upsilon(t)}_{\mathcal L(\mathsf X)} \leq C e^{-\gamma t}\) for every \(t \geq 0\).

Given a Banach space \(\mathsf X\), its dual space \(\mathcal L(\mathsf X, \mathbb C)\) is denoted by \(\mathsf X^\prime\), and we write \(\scalprod{\cdot, \cdot}_{\mathsf X^\prime \times \mathsf X}\) for the dual pairing in \(\mathsf X^\prime \times \mathsf X\), the index \(\mathsf X^\prime \times \mathsf X\) being omitted when clear from the context. For \(r \in [1, +\infty)\), we always identify the dual of \(L^r((a, b), \mathbb C^d)\) with \(L^{r^\prime}((a, b), \mathbb C^d)\), where \(r^\prime \in (1, +\infty]\) is the conjugate exponent of \(r\), satisfying \(\frac{1}{r} + \frac{1}{r^\prime} = 1\). For simplicity, we denote the duality pairing on these spaces by $\langle \cdot,\cdot \rangle_{L^r}$, namely,
\[
\langle x,y \rangle_{L^r} = \int_{a}^b \langle x(\theta), y(\theta) \rangle_{\mathbb{C}^d} \, d\theta, \quad x \in L^{r}((a, b), \mathbb C^d), \; y \in L^{r^\prime}((a, b), \mathbb C^d).
\]

Given an interval \(I \subset \mathbb R\), we denote by \(C_c^\infty(I, \mathbb C^d)\) the set of functions of class \(C^\infty\) defined over \(I\) and taking values in \(\mathbb C^d\) whose support is a compact subset of \(I\). When \(I\) is open, we assume that elements of \(C_c^\infty(I, \mathbb C^d)\) are extended by \(0\) to elements of \(C_c^\infty(\mathbb R, \mathbb C^d)\) whenever needed.

\section{Abstract representation as a linear control system}
\label{sec:representation}

The main goal of this section is to rewrite \eqref{main_system} in \(\mathsf X^r\) in the form of an abstract linear control system
\begin{equation}
\label{eq:abstract}
\left\{
\begin{aligned}
z^\prime(t) & = F_r z(t) + G_r u(t), \\
z(0) & = z_0,
\end{aligned}
\right.
\end{equation}
where \(F_r\) is the generator of a \(C^0\)-semigroup on a suitable functional space, \(G_r\) is the control operator, and the state \(z\) at a given time \(t\) is in correspondence with the history function \(x_t\) of \eqref{main_system}. In the sequel, we shall consider \emph{mild} solutions for \eqref{eq:abstract} as defined below in Section~\ref{sec:abstract}. 

\subsection{Previous results}

To relate \eqref{main_system} with a system of the form \eqref{eq:abstract}, we first recall some previously obtained results for \eqref{main_system}. The following statement gathers the results presented in \cite[Propositions~2.2 and 2.7 and Remark~2.3]{Mazanti2017Relative} (see also \cite[Propositions~2.2 and 2.4, Remark~2.5, and Equation~(2.5)]{Chitour2020Approximate}). A general realization theory for infinite-dimensional linear input/output systems was developed by Y.~Yamamoto~\cite{yamamoto1981realization}, with a realization framework for pseudo-rational systems presented in \cite{YamamotoRealization}. For the difference delay equations considered here, however, the state-space representation can be constructed directly from the solution operators on the history space, without invoking this general theory.

\begin{proposition}
\label{prop:exist-unique-formula}
Let \(r \in [1, +\infty]\), \(x_0 \in \mathsf X^r\), and \(u \in \mathsf U^r\). Then \eqref{main_system} admits a unique solution with initial condition \(x_0\) and control \(u\), \emph{i.e.}, there exists a unique function \(x\in L^r_{\mathrm{loc}}((-\Lambda_N, +\infty), \mathbb C^d)\) satisfying \eqref{main_system} almost everywhere on \((-\Lambda_N, +\infty)\).

In addition, there exist families of operators \((\Upsilon_r(t))_{t \geq 0}\) and \((E_r(t))_{t \geq 0}\) in the spaces \(\mathcal L(\mathsf X^r)\) and \(\mathcal L(L^r((0, t), \mathbb C^m), \mathsf X^r)\), respectively, such that \((\Upsilon_r(t))_{t \geq 0}\) is a semigroup and, for every \(x_0 \in \mathsf X^r\) and \(u \in \mathsf U^r\), the history function \(t \mapsto x_t\) associated with the unique solution \(x\) of \eqref{main_system} satisfies
\begin{equation}
\label{eq:solution}
x_t = \Upsilon_r(t) x_0 + E_r(t) u_{\rvert [0, t]}, \qquad t \geq 0.
\end{equation}

Moreover, for \(t \in [0, \Lambda_1)\), we have, for almost every \(\theta \in (-\Lambda_N, 0)\),
\begin{subequations}
\label{eq:Upsilon-E}
\begin{align}
\label{eq:Upsilon}
(\Upsilon_r(t) x_0)(\theta) & = \begin{dcases*}
\sum_{j = 1}^N A_j x_0(t + \theta - \Lambda_j) & if \(0 \leq t + \theta < \Lambda_1\), \\
x_0(t + \theta) & if \(t + \theta < 0\),
\end{dcases*} \\
\label{eq:E}
(E_r(t) u)(\theta) & = \begin{dcases*}
B u(t + \theta) & if \(0 \leq t + \theta < \Lambda_1\), \\
0 & if \(t + \theta < 0\),
\end{dcases*}
\end{align}
\end{subequations}
and \((\Upsilon_r(t))_{t \geq 0}\) is a \(C^0\)-semigroup if \(r \in [1, +\infty)\).
\end{proposition}

\begin{remark}
Although we refer to \cite{Chitour2020Approximate, Mazanti2017Relative} for the proof of Proposition~\ref{prop:exist-unique-formula}, most of its consequences are actually straightforward to establish. Indeed, the existence and uniqueness of solutions of \eqref{main_system} can be easily established by remarking that \eqref{main_system} already gives an explicit formula for the solution \(x(t)\) whenever \(t \in [0, \Lambda_1)\) (yielding in particular \eqref{eq:Upsilon-E}), and this procedure can be iterated on intervals of length \(\Lambda_1\) to obtain the existence and uniqueness of a solution defined in \((-\Lambda_N, +\infty)\), which can be easily seen to belong to \(L^r_{\mathrm{loc}}((-\Lambda_N, +\infty), \mathbb C^d)\) by the same iterative argument. The identity \eqref{eq:solution} is a simple consequence of the linearity of the equation, and the fact that \((\Upsilon_r(t))_{t \geq 0}\) is a semigroup of linear operators comes from the fact that, in the absence of control, \eqref{main_system} is autonomous. The works \cite{Chitour2020Approximate, Mazanti2017Relative} also provide explicit formulas for \(\Upsilon_r(t)\) and \(E_r(t)\) for any \(t \geq 0\), which allow one to easily obtain that these linear operators are continuous. Finally, the fact that \((\Upsilon_r(t))_{t \geq 0}\) is a \(C^0\)-semigroup if \(r \in [1, +\infty)\) follows easily from \eqref{eq:Upsilon} and the continuity of translations in \(L^r\) for finite \(r\).
\end{remark}

\subsection{The operator \texorpdfstring{\(F_r\)}{Fr} and its adjoint}
\label{sec:Fr}

In the light of Proposition~\ref{prop:exist-unique-formula}, the natural candidate for the operator \(F_r\) appearing in \eqref{eq:abstract} is the generator of the semigroup \((\Upsilon_r(t))_{t \geq 0}\). The next result characterizes this operator. Its proof follows classical arguments in functional analysis (see, e.g., \cite[\S~II.2.10, Proposition~1]{Engel2001} for the case of translation semigroups on the whole real line).

\begin{proposition}
\label{prop:generator-Fr}
Let \(r \in [1, +\infty)\) and \((\Upsilon_r(t))_{t \geq 0}\) be the \(C^0\)-semigroup from Proposition~\ref{prop:exist-unique-formula}. The generator of \((\Upsilon_r(t))_{t \geq 0}\) is the operator \(F_r\colon D(F_r) \subset \mathsf X^r \to \mathsf X^r\) defined by
\begin{equation}
\label{eq:def-Fr}
\begin{aligned}
D(F_r) & = \left\{z \in W^{1, r}((-\Lambda_N, 0), \mathbb C^d) \suchthat z(0) = \sum_{j = 1}^N A_j z(-\Lambda_j)\right\}, \\
F_r z & = z^\prime.
\end{aligned}
\end{equation}
\end{proposition}

\begin{proof}
Since \((\Upsilon_r(t))_{t \geq 0}\) is a \(C^0\)-semigroup, it admits a generator. Let \(F_r\colon D(F_r) \subset \mathsf X^r \to \mathsf X^r\) be the operator defined in \eqref{eq:def-Fr} and \(\mathcal F_r\colon D(\mathcal F_r) \subset \mathsf X^r \to \mathsf X^r\) be the generator of \((\Upsilon_r(t))_{t \geq 0}\).

We first show that \(\mathcal F_r \subset F_r\), i.e., that \(D(\mathcal F_r) \subset D(F_r)\) and that \(\mathcal F_r z = F_r z\) for every \(z \in D(\mathcal F_r)\). Take \(z \in D(\mathcal F_r)\) and let \(w = \mathcal F_r z \in \mathsf X^r\). Take \(\varphi \in C^\infty_c((-\Lambda_N, 0), \mathbb C^d)\). By definition of the generator and continuity of the duality bracket, we have
\begin{align*}
\scalprod{w, \varphi}_{L^r} & = \lim_{h \to 0^+} \scalprod*{\frac{\Upsilon_r(h) z - z}{h}, \varphi}_{L^r} \\ 
& = \lim_{h \to 0^+} \frac{1}{h} \left[\int_{-\Lambda_N}^0 \scalprod*{(\Upsilon_r(h) z)(\theta), \varphi(\theta)}_{\mathbb C^d} \,d\theta - \int_{-\Lambda_N}^0 \scalprod*{z(\theta), \varphi(\theta)}_{\mathbb C^d} \,d\theta\right]. 
\end{align*}
Since \(\varphi\) is compactly supported in \((-\Lambda_N, 0)\), for \(h > 0\) small enough, we have
\[
\int_{-\Lambda_N}^0 \scalprod*{(\Upsilon_r(h) z)(\theta), \varphi(\theta)}_{\mathbb C^d} \,d\theta
 = \int_{-\Lambda_N}^{0} \scalprod*{z(\theta), \varphi(\theta - h)}_{\mathbb C^d} \,d\theta,
\]
so that, by Lebesgue's dominated convergence theorem,
\[
\scalprod{w, \varphi}_{L^r} = \lim_{h \to 0^+} \int_{-\Lambda_N}^0 \scalprod*{z(\theta), \frac{\varphi(\theta - h) - \varphi(\theta)}{h}}_{\mathbb C^d} \,d\theta = - \scalprod{z, \varphi^\prime}_{L^r}.
\]
Since $z\in \mathsf X^r$ and $w\in \mathsf X ^r$, we deduce that \(z \in W^{1, r}((-\Lambda_N, 0), \mathbb C^d)\) and \(z^\prime = w\).

Now, let \(\varphi \in C^1_c((-\Lambda_N, 0], \mathbb C^d)\). Proceeding as before and using the Sobolev injection of \(W^{1, r}((-\Lambda_N, 0), \mathbb C^d)\) into \(C([-\Lambda_N, 0], \mathbb C^d)\), we get
\begin{align*}
\scalprod{z^\prime, \varphi}_{L^r} & = \lim_{h \to 0^+} \scalprod*{\frac{\Upsilon_r(h) z - z}{h}, \varphi}_{L^r} \\
& = \lim_{h \to 0^+} \Biggl[\int_{-\Lambda_N}^0 \scalprod*{z(\theta), \frac{\varphi(\theta - h) - \varphi(\theta)}{h}}_{\mathbb C^d} \,d\theta \\
& \hphantom{\displaystyle = \lim_{h \to 0^+} \Biggl[} + \frac{1}{h}\int_{-h}^0 \scalprod*{\sum_{j = 1}^N A_j z(h + \theta - \Lambda_j), \varphi(\theta)}_{\mathbb C^d} \,d\theta\Biggr] \\
& = -\scalprod{z, \varphi^\prime}_{L^r} + \scalprod*{\sum_{j = 1}^N A_j z(-\Lambda_j), \varphi(0)}_{\mathbb C^d}.
\end{align*}
On the other hand, integrating by parts, we get that \(\scalprod{z^\prime, \varphi}_{L^r} = -\scalprod{z, \varphi^\prime}_{L^r} + \scalprod{z(0), \varphi(0)}_{\mathbb C^d}\), so that
\[
\scalprod{z(0), \varphi(0)}_{\mathbb C^d} = \scalprod*{\sum_{j = 1}^N A_j z(-\Lambda_j), \varphi(0)}_{\mathbb C^d}.
\]
Since the above identity holds for every \(\varphi \in C^1_c((-\Lambda_N, 0], \mathbb C^d)\), we deduce that \(z(0) = \sum_{j = 1}^N A_j z(-\Lambda_j)\), and thus \(z \in D(F_r)\). So \(\mathcal F_r \subset F_r\).

To prove that \(\mathcal F_r = F_r\), it suffices now to show that their respective resolvent sets \(\rho(\mathcal F_r)\) and \(\rho(F_r)\) have a non-empty intersection (see, e.g., \cite[\S~IV.1.21, Exercise (5)]{Engel2001}). Since \(\mathcal F_r\) is the generator of \((\Upsilon_r(t))_{t \geq 0}\), there exists \(\omega \in \mathbb R\) such that \((\omega, +\infty) \subset \rho(\mathcal F_r)\). It thus suffices to prove that, for a real number \(\lambda\) large enough, we have \(\lambda \in \rho(F_r)\).
Let \(\lambda \in \mathbb R\) and consider the operator
\(\lambda I - F_r\). Assume that \(\lambda\) is large enough
so that the matrix
\[
A_\lambda = I_d - \sum_{j=1}^N e^{-\lambda \Lambda_j} A_j
\]
is invertible, which is possible since \(A_\lambda \to I_d\)
as \(\lambda \to +\infty\).
Define \(B_\lambda \colon \mathsf X^r \to \mathsf X^r\)
by setting, for \(w \in \mathsf X^r\) and a.e.\ \(\theta \in (-\Lambda_N, 0)\),
\[
(B_\lambda w)(\theta)
=
e^{\lambda \theta} A_\lambda^{-1}
\sum_{j=1}^N \int_{-\Lambda_j}^0
e^{-\lambda(s+\Lambda_j)} A_j w(s)\,ds
+
\int_\theta^0 e^{-\lambda(s-\theta)} w(s)\,ds.
\]
For \(w \in \mathsf X^r\), set \(z = B_\lambda w\).
The above formula gives \(z \in W^{1,r}((-\Lambda_N,0),\mathbb C^d\)),
\(z' = \lambda z-w\), and
\(\|z\|_{\mathsf X^r} \le C_\lambda \|w\|_{\mathsf X^r}\) for some constant \(C_\lambda > 0\) independent of \(w\).
Moreover,
\[
z(0)-\sum_{j=1}^N A_j z(-\Lambda_j)
=
A_\lambda z(0)
-\sum_{j=1}^N \int_{-\Lambda_j}^0
e^{-\lambda(s+\Lambda_j)} A_j w(s)\,ds
=0.
\]
It follows that \(B_\lambda w \in D(F_r)\) for every
\(w \in \mathsf X^r\), \(B_\lambda \in \mathcal L(\mathsf X^r)\),
and \((\lambda I-F_r)B_\lambda w=w\) for every
\(w \in \mathsf X^r\).
Finally, if \(z \in D(F_r)\) satisfies \((\lambda I-F_r)z=0\),
then \(z(\theta)=e^{\lambda\theta}z(0)\), and the boundary
condition yields \(A_\lambda z(0)=0\). Since \(A_\lambda\)
is invertible, \(z=0\), so \(\lambda I-F_r\) is injective.
Hence \(\lambda \in \rho(F_r)\), yielding the conclusion.
\end{proof}

For later use, it is crucial to characterize the adjoint of \(F_r\), which we now compute. For that purpose, set $\Lambda_0 = 0$ and $I_j = (-\Lambda_{j+1}, -\Lambda_j)$ for $j \in \{0, \dots, N-1\}$. As we shall need only \(r \in (1, +\infty)\) in the sequel, we also make this assumption from now on.

\begin{proposition}
\label{prop:adjoint_F}
Let \(r \in (1, +\infty)\) and \(F_r\colon D(F_r) \subset \mathsf X^r \to \mathsf X^r\) be the operator defined in \eqref{eq:def-Fr}. Then the adjoint of \(F_r\) is the operator \(F_r^\ast \colon D(F_r^\ast) \subset \mathsf X^{r'} \to \mathsf X^{r'}\) defined by
\begin{align}
D(F^*_r) = {} & \Bigl\{y \in \mathsf X^{r'} \mathop{}\Big\vert\mathop{} y_{|I_j} \in W^{1,r'}(I_j, \mathbb C^d)\ \text{for all }j \in \{0,\dots,N-1\}, \label{DFe} \\
& y(-\Lambda_j^+) - y(-\Lambda_j^-) = A_j^*\,y(0^-) \text{ for } j\in\{1,\dots,N-1\}, \notag \\
& y(-\Lambda_N^+)=A_N^*\,y(0^-) \Bigr\}, \notag \\
(F^*_r y)(\theta) = {} & -y_{| I_j}'(\theta) \text{ for every } j \in \{0, \dotsc, N-1\} \text{ and a.e.\ } \theta \in I_j. \label{eq:F-r-ast}
\end{align}
\end{proposition}

\begin{proof}
Let \(\mathcal D\) be the set defined in the right-hand side of \eqref{DFe}.

We first show that \(D(F_r^\ast) \subset \mathcal D\) and that \(F_r^\ast\) is given by \eqref{eq:F-r-ast} on \(D(F_r^\ast)\). Let $y\in D(F^*_r)$ and set $g = F^*_ry\in \mathsf X^{r'}$. By definition of the adjoint,
\begin{equation}
\label{eq:adjoint_identity}
\langle F_rx,y\rangle_{L^r}
=
\langle x,g\rangle_{L^r} \qquad \text{ for all } x\in D(F_r).
\end{equation}
Let us start by showing that $y$ is $W^{1,r'}$ on each interval $I_j$.
Fix $j\in\{0,\dots,N-1\}$ and let $\varphi\in C_c^\infty(I_j,\mathbb C^d) \subset C_c^\infty((-\Lambda_N,0),\mathbb C^d)\subset W^{1,r}((-\Lambda_N,0),\mathbb C^d)$. Clearly, $\varphi \in D(F_r)$. Applying \eqref{eq:adjoint_identity} with $x = \varphi$, we obtain
$\int_{I_j}\langle \varphi'(\theta),y(\theta)\rangle_{\mathbb C^d}\,d\theta
=
\int_{I_j}\langle \varphi(\theta),g(\theta)\rangle_{\mathbb C^d}\,d\theta,$
\textit{i.e.}, $y'=-g$ in the distributional sense on $I_j$. Since $g\in L^{r'}((-\Lambda_N, 0), \allowbreak\mathbb C^d)$, we deduce that
$y_{|I_j}\in W^{1,r'}(I_j,\mathbb C^d)$ and $y_{|I_j}'=-g\in L^{r'}(I_j,\mathbb C^d)$, so that \eqref{eq:F-r-ast} is verified.
In addition, $y$ is piecewise $W^{1,r'}$, so that the one-sided traces $y(0^-)$, $y(-\Lambda_N^+)$ and $y(-\Lambda_j^\pm)$, $j \in \{1,\dots,N-1\}$, are well defined.
Using~\eqref{eq:adjoint_identity} with any $x\in D(F_r)$, and recalling that
$F_rx=x'$ and $g=-y'$ on each $I_j$, we obtain 
\[
0=
\langle F_rx,y\rangle_{L^{r}}-\langle x,g\rangle_{L^{r}}=
\sum_{j=0}^{N-1}\int_{I_j}
\Bigl(
\langle x'(\theta),y(\theta)\rangle_{\mathbb C^d}
+
\langle x(\theta),y'(\theta)\rangle_{\mathbb C^d}
\Bigr)\,d\theta.
\]
Integrating by parts on each interval $I_j$ and summing up, we get
\begin{multline*}
0 = \langle x(0),y(0^-)\rangle_{\mathbb C^d}
-\langle x(-\Lambda_N),y(-\Lambda_N^+)\rangle_{\mathbb C^d} \\ +\sum_{j=1}^{N-1}
\langle x(-\Lambda_j),\,y(-\Lambda_j^-)-y(-\Lambda_j^+)\rangle_{\mathbb C^d}.
\end{multline*}
Since $x \in D(F_r)$, we have $x(0)=\sum_{j=1}^N A_jx(-\Lambda_j)$,
and then
\begin{align*}
0 = {} &
\sum_{j=1}^N \langle A_jx(-\Lambda_j),y(0^-)\rangle_{\mathbb C^d}
-\langle x(-\Lambda_N),y(-\Lambda_N^+)\rangle_{\mathbb C^d} \\
& + \sum_{j=1}^{N-1}
\langle x(-\Lambda_j),\,y(-\Lambda_j^-)-y(-\Lambda_j^+)\rangle_{\mathbb C^d} \\
{} = {} & \sum_{j=1}^{N-1}
\left\langle
x(-\Lambda_j),
A_j^*y(0^-)+y(-\Lambda_j^-)-y(-\Lambda_j^+)
\right\rangle_{\mathbb C^d} \\
& +
\left\langle
x(-\Lambda_N),
A_N^*y(0^-)-y(-\Lambda_N^+)
\right\rangle_{\mathbb C^d}.
\end{align*}
Since this identity holds for every \(x \in D(F_r)\), we deduce that
$A_j^*y(0^-)+y(-\Lambda_j^-)-y(-\Lambda_j^+)=0$ for $j \in \{1,\dots,N-1\}$ and
$A_N^*y(0^-)-y(-\Lambda_N^+)=0$. Hence \(D(F_r^\ast) \subset \mathcal D\).

To prove that \(\mathcal D \subset D(F_r^\ast)\), take \(y \in \mathcal D\) and define $g\in L^{r'}((-\Lambda_N,0),\mathbb C^d)$ by
$g(\theta) = -y_{| I_j}'(\theta)$ for every \(j \in \{0, \dotsc, N-1\}\) and a.e.\ $\theta\in I_j$. Let \(x \in D(F_r)\). Integrating by parts on each interval $I_j$ yields
\begin{align*}
\langle F_rx,y\rangle_{L^{r}} = {}
& \sum_{j=0}^{N-1}\int_{I_j}\langle x'(\theta),y(\theta)\rangle_{\mathbb C^d}\,d\theta \\
{} = {} & -\sum_{j=0}^{N-1}\int_{I_j}\langle x(\theta),y'(\theta)\rangle_{\mathbb C^d}\,d\theta \\
&
+\langle x(0),y(0^-)\rangle_{\mathbb C^d}
-\langle x(-\Lambda_N),y(-\Lambda_N^+)\rangle_{\mathbb C^d} \\
&
+\sum_{j=1}^{N-1}
\langle x(-\Lambda_j),\,y(-\Lambda_j^-)-y(-\Lambda_j^+)\rangle_{\mathbb C^d}.
\end{align*}
Using the fact that $x(0)=\sum_{j=1}^N A_jx(-\Lambda_j)$ and the jump conditions on \(y \in \mathcal D\), all boundary terms in the above expression cancel out, and we obtain
\begin{align*}
\langle F_rx,y\rangle_{L^{r}}
&=
-\sum_{j=0}^{N-1}\int_{I_j}\langle x(\theta),y'(\theta)\rangle_{\mathbb C^d}\,d\theta=
\int_{-\Lambda_N}^0 \langle x(\theta),g(\theta)\rangle_{\mathbb C^d}\,d\theta =
\langle x,g\rangle_{L^{r}}.
\end{align*}
Hence, by definition of the adjoint, we have $y\in D(F^*_r)$ and $F^*_ry=g$, as required.
\end{proof}

\begin{remark}
\label{remk:adjoint-semigroup}
Let \(r \in (1, +\infty)\). Since \(\mathsf X^r\) is a reflexive space, the adjoint semigroup \((\Upsilon_r(t)^\ast)_{t \geq 0}\) is a \(C^0\)-semigroup in \(\mathsf X^{r'}\) and its generator is \(F_r^\ast\) (see, e.g., \cite[Chapter~I, \S~5.14]{Engel2001}).
\end{remark}

In the sequel, we always consider that $D(F^*_r)$ is endowed with its graph norm, given by $\|x\|_{D(F^*_r)}= \|x\|_{\mathsf X^{r'}} + \|F_r^*x\|_{\mathsf X^{r'}}$ for \(x \in D(F_r^\ast)\), which makes it a Banach space since $F^*_r$ is closed. We denote by $D(F^*_r)'$ the dual space of $D(F^*_r)$. We shall also need to work with the extrapolation space \(\mathsf X_{-1}^{r}\) of order \(-1\) defined as the completion of \(\mathsf X^r\) with respect to the norm \(x \mapsto \norm{(\lambda I - F_r)^{-1} x}_{\mathsf X^r}\) (for a fixed $\lambda \in \rho(F_r)$), and we refer to \cite[Section~II.5.a]{Engel2001} for properties of extrapolation spaces used in the sequel. Since $\mathsf X^{r}$ is reflexive, the space \(\mathsf X^{r}_{-1}\) can be canonically identified with $D(F_r^*)'$. Moreover, the operator \(F_r\) extends to an unbounded operator in \(\mathsf X_{-1}^r\) with domain \(\mathsf X^r\), which we still denote by \(F_r\), and the semigroup $(\Upsilon_r(t))_{t\geq 0}$ extends to a $C_0$-semigroup on \(\mathsf X_{-1}^r\), which we still denote by $(\Upsilon_r(t))_{t\geq 0}$, and whose generator is the extended operator \(F_r\). In addition, this extension satisfies \(F_r \in \mathcal L(\mathsf X^r, \mathsf X_{-1}^{r})\).

With these definitions, one can already identify \eqref{main_system} and \eqref{eq:abstract} in the absence of control. Indeed, the Cauchy problem \(z^\prime(t) = F_r z(t)\) in \(\mathsf X_{-1}^r\) with initial condition \(z(0) = z_0 \in \mathsf X^r\) admits a unique solution \(z \in C^0(\mathbb R_+, \mathsf X^r) \cap C^1(\mathbb R_+, \mathsf X_{-1}^r)\), given by \(z(t) = \Upsilon_r(t) z_0\) for \(t \in \mathbb R_+\), and, by Proposition~\ref{prop:exist-unique-formula}, this is precisely the history function associated with the unique solution of \eqref{main_system} with initial condition \(x_0 = z_0\) and null control.

\subsection{The operator \texorpdfstring{$G_r$}{Gr} and its adjoint}

We next provide an operator \(G_r \in \mathcal L(\mathbb C^m, D(F_r^\ast)')\) that will enable us to precisely define mild solutions of \eqref{eq:abstract} and relate them to solutions of \eqref{main_system}. We define \(G_r\) by setting, for \(u \in \mathbb C^m\),
\begin{equation}
\label{defG}
G_r u = \scalprod{B u, \delta_{0} (\cdot)}_{\mathbb C^d} \text{ in } D(F^*_r)',
\end{equation}
where \(\delta_0\colon D(F^*_r) \to \mathbb C^d\) is the linear map given by \(\delta_0 y = y(0^-)\) for \(y \in D(F_r^\ast)\). Using standard trace estimates in \(W^{1, r'}\), \(\delta_0\) is continuous, and hence \(G_r \in \mathcal L(\mathbb C^m, D(F_r^\ast)')\). The adjoint operator \(G_r^\ast\) is the linear operator defined, for \(u\in\mathbb C^m\) and \(y \in D(F^*_r)\), by the duality formula
\begin{equation}
\label{defGadj}
\scalprod{G_r u, y}_{D(F^*_r)' \times D(F^*_r)} = 
\scalprod{u,G_r^\ast y}_{\mathbb C^m}=
\scalprod{Bu, y(0^-)}_{\mathbb C^d},
\end{equation}
and is simply given by
\begin{equation}
\label{eq:def-Gr}
G_r^\ast y = B^\ast y(0^-),\qquad
y \in D(F_r^\ast).
\end{equation}

\subsection{Equivalence between Systems~(\ref{main_system}) and (\ref{eq:abstract})}\label{sec:abstract}

Using the previous definition of \(G_r\), we are in a position
to relate Systems \eqref{main_system} and \eqref{eq:abstract}.
Given \(z_0 \in \mathsf X^r\) and \(u \in \mathsf U^r\),
a \emph{mild solution} of \eqref{eq:abstract} with initial condition
\(z_0\) and control \(u\) is a function \(z \in C^0(\mathbb R_+,\mathsf X^r)\)
satisfying
\begin{equation}\label{eq:mild_solution_abstract}
z(t)
=
\Upsilon_r(t)z_0
+
\int_0^t \Upsilon_r(t-s)G_r u(s)\,ds,
\qquad t \geq 0,
\end{equation}
where the integral is taken in \(\mathsf X^r_{-1}\). Note also that, for \(t > 0\), the integral term in \eqref{eq:mild_solution_abstract} defines an operator from \(L^r((0, t), \mathbb C^m)\) to \(\mathsf X_{-1}^r\). We recall that a control operator \(G_r\) is said to be \emph{admissible} for \(F_r\) (or, equivalently, for \((\Upsilon_r(t))_{t \geq 0}\)) if the operator defined by the integral term of \eqref{eq:mild_solution_abstract} actually belongs to \(\mathcal L(L^r((0, t), \mathbb C^m), \mathsf X^r)\). The theory of admissible control and observation operators in \(L^r\)-spaces goes back in particular to Weiss \cite{Weiss1989Control, Weiss1989Observation}; see also the monograph by
Staffans \cite[Chapter~10]{Staffans2005} for the \(L^r\)-well-posed
systems framework, and the survey \cite{JacobPartington2004} for an
overview of admissibility questions.

Our main result on the equivalence between Systems~\eqref{main_system} and \eqref{eq:abstract} is the following.

\begin{theorem}
\label{th_system_abstract}
Let \(r \in (1, +\infty)\), \(F_r\) be defined by \eqref{eq:def-Fr}, \((\Upsilon_r(t))_{t \geq 0}\) be the semigroup generated by \(F_r\) defined in Proposition~\ref{prop:exist-unique-formula}, and \(G_r\) be defined by \eqref{defG}. Take \(x_0 \in \mathsf X^r\), \(u \in \mathsf U^r\). Then \(z\) is a mild solution of \eqref{eq:abstract} starting at $x_0$ if and only if \(z(t) = x_t\), where \(t \mapsto x_t\) is the history function associated with the unique solution \(x\) of \eqref{main_system} starting at $x_0$. As a consequence, \(G_r\) is admissible for \(F_r\).
\end{theorem}

\begin{proof}
By uniqueness of solutions and the fact that System~\eqref{main_system} is autonomous, the families
\((\Upsilon_r(t))_{t\geq0}\) and \((E_r(t))_{t\geq0}\)
satisfy the composition property of
\cite[Definition~2.1]{Weiss1989Control}.
Hence, by \cite[Theorem~3.9]{Weiss1989Control}, there exists
a unique admissible control operator
\(\widetilde G_r \in
\mathcal L(\mathbb C^m,\mathsf X^r_{-1})\)
representing \(E_r(t)\).

For \(v \in \mathbb C^m\), let \(\mathbf v\) denote
the constant control equal to \(v\).
Equation~(3.6) in \cite{Weiss1989Control} gives
\[
\widetilde G_r v
=
\lim_{h\downarrow0}
\frac{1}{h}E_r(h)(\mathbf v|_{[0,h]})
\quad\text{in }\mathsf X^r_{-1}.
\]
Using \eqref{eq:E}, we obtain, for every \(y \in D(F_r^*)\),
\[
\begin{aligned}
\langle \widetilde G_r v,y\rangle
&=
\lim_{h\downarrow0}
\frac{1}{h}\int_{-h}^0
\langle Bv,y(\theta)\rangle_{\mathbb C^d}\,d\theta\\
&=
\langle Bv,y(0^-)\rangle_{\mathbb C^d}
=
\langle G_r v,y\rangle.
\end{aligned}
\]
Thus \(\widetilde G_r=G_r\), and therefore
\begin{equation}\label{eq:equality-Et-convo}
E_r(t)u|_{[0,t]}
=
\int_0^t \Upsilon_r(t-s)G_r u(s)\,ds,
\qquad u \in \mathsf U^r,\quad t\geq0.
\end{equation}
The equivalence follows from \eqref{eq:solution}
and \eqref{eq:mild_solution_abstract}, while the admissibility
of \(G_r\) follows from that of \(\widetilde G_r\).
\end{proof}

\section{Exact controllability}
\label{sec:controllability}
Having established the abstract formulation \eqref{eq:abstract} of System~\eqref{main_system}, we are now in a position to investigate its control-theoretic properties, and more specifically, in this section, the notion of exact controllability, whose definition we provide next.

\begin{definition}
\label{def:controllability}
Let \(r \in [1, +\infty]\) and \(T > 0\). System~\eqref{main_system} is said to be \emph{exactly controllable in $\mathsf X^r$ in time $T$} if, for every initial condition $x_0 \in \mathsf{X}^r$ and every target $\psi \in \mathsf{X}^r$, there exists a control $u \in \mathsf{U}^r$ such that the history function \(t \mapsto x_t\) associated with the unique solution \(x\) of \eqref{main_system} with initial condition \(x_0\) and control \(u\) satisfies \(x_T = \psi\).
\end{definition}

\begin{remark}
\label{remk:Cd-Rd}
The choice of considering complex-valued matrices, states and control in this paper is motivated by the fact that the controllability criteria we shall establish are more naturally stated in such a framework, which also has the advantage of simplifying some proofs (see, e.g., the proof of Proposition~\ref{prop:Lr_Hautus} below, which makes use of the fact that one works with complex Banach spaces). This choice is not restrictive since, when the matrices \(A_1, \dotsc, A_N\) and \(B\) have real coefficients, the dynamics of the real and imaginary parts of the state \(x\) and the control \(u\) of \eqref{main_system} are independent. Hence, if \(x_{\rm ini}\) and \(x_{\rm fin}\) take values in \(\mathbb R^d\) and \(u\in L^r((0, T),\mathbb{C}^m)\) is a control steering \(x_{\rm ini}\) to \(x_{\rm fin}\) in time \(T\), then \(\Re(u) \in L^r((0, T),\mathbb{R}^m)\) (where the real part is taken componentwise) also steers \(x_{\rm ini}\) to \(x_{\rm fin}\) in time \(T\). Conversely, if \(x_{\rm ini}\) and \(x_{\rm fin}\) take values in \(\mathbb C^d\) and \(u_1, u_2\in L^r((0, T),\mathbb{R}^m)\) are controls steering \(\Re(x_{\rm ini})\) to \(\Re(x_{\rm fin})\) and \(\Im(x_{\rm ini})\) to \(\Im(x_{\rm fin})\), respectively, in time \(T\), then \(u = u_1 + i u_2\) steers \(x_{\rm ini}\) to \(x_{\rm fin}\) in time \(T\). In particular, the main results of this paper remain true when we replace \(\mathbb C^d\) and \(\mathbb C^m\) by \(\mathbb R^d\) and \(\mathbb R^m\), respectively.
\end{remark}

While the notion of exact controllability depends on \(T\), we have the following result relating exact controllability for different times, which is an immediate consequence of \cite[Theorem~4.7]{chitour:hal-03827918}\footnote{\cite[Theorem~4.7]{chitour:hal-03827918} is stated only for finite \(r\), but the argument of its proof also covers the case \(r = +\infty\).}.

\begin{proposition}
\label{prop:control-equivalence-time}
Let \(r \in [1, +\infty]\). Then the following assertions are equivalent.
\begin{enumerate}
\item There exists \(T > 0\) such that System~\eqref{main_system} is exactly controllable in \(\mathsf X^r\) in time \(T\).
\item\label{item:contr-dLambdaN} System~\eqref{main_system} is exactly controllable in \(\mathsf X^r\) in time \(d \Lambda_N\).
\item\label{item:contr-geq-dLambdaN} System~\eqref{main_system} is exactly controllable in \(\mathsf X^r\) in time \(T\) for every \(T \in [d \Lambda_N, +\infty)\).
\end{enumerate}
\end{proposition}

In the sequel, we say that System~\eqref{main_system} is \emph{exactly controllable in \(\mathsf X^r\)} if any of the equivalent assertions from Proposition~\ref{prop:control-equivalence-time} is satisfied.

Note that, by Theorem~\ref{th_system_abstract}, for \(r \in (1, +\infty)\), exact controllability of \eqref{main_system} in \(\mathsf X^r\) in time \(T\) is equivalent to the corresponding property for \eqref{eq:abstract}, namely, that, for every \(z_0, \psi \in \mathsf X^r\), there exists a control \(u \in \mathsf U^r\) such that the unique mild solution \(z\) of \eqref{eq:abstract} with initial condition \(z_0\) and control \(u\) satisfies \(z(T) = \psi\).

In order to study the controllability properties of \eqref{main_system}, we introduce its \emph{characteristic matrix}, defined as the holomorphic function \(H \colon \mathbb C \to \mathbb C^{d \times d}\) given by
\begin{equation}\label{hqp}
H(p) = I_d - \sum_{j=1}^{N} e^{-p\Lambda_j} A_j, \qquad p \in \mathbb{C}.
\end{equation}
We shall denote by \(H(\mathbb C) \subset \mathbb C^{d \times d}\) the range of \(H\), and by \(\overline{H(\mathbb C)}\) its closure. The characteristic matrix has long been used in the stability analysis of \eqref{main_system} (see, e.g., \cite[Chapter~9]{Hale1993Introduction}) and it can be obtained by taking the Laplace transform of \eqref{main_system}. Its properties have also been related to controllability properties of \eqref{main_system} in \cite{chitour:hal-03827918, fueyo:hal-04830757, Corona_fueyo} and, in particular, we have the following result.

\begin{proposition}[{\cite[Theorems~3.3 and 3.5]{fueyo:hal-04830757}}]
\label{prop_con}
Let $r \in [1, +\infty)$ and assume that the following conditions hold:
\begin{enumerate}[1)]
\item \label{assumption1_prev}
$\rank \left[M, B\right]=d$ for every $M \in \overline{H(\mathbb{C})}$,
\item \label{assumption2_prev}
$\rank \left[A_N, B\right]=d$.
\end{enumerate}
Then System~\eqref{main_system} is exactly controllable in $\mathsf{X}^r$. Moreover, when $r=1$, these conditions are also necessary for exact controllability.
\end{proposition}

Our aim in the remainder of this section is to prove that conditions \ref{assumption1_prev} and \ref{assumption2_prev} from Proposition~\ref{prop_con} are necessary conditions for exact controllability also in the case \(r \in (1, +\infty)\), providing thus a positive answer to \cite[Conjecture~3.8]{chitour:hal-03827918}. More precisely, our main result in this section is the following.

\begin{theorem}[Characterization of exact controllability]
\label{th:iff_controllability}
Let $r \in [1, +\infty)$. Then System~\eqref{main_system} is exactly controllable in $\mathsf{X}^r$ if and only if the following conditions hold:
\begin{enumerate}[1)]
\item\label{item:cond-rank-H} $\rank \left[M, B\right] = d$ for every $M \in \overline{H(\mathbb{C})}$,
\item\label{item:cond-rank-AN} $\rank \left[A_N, B\right] = d$.
\end{enumerate}
\end{theorem}

\begin{remark}
From Theorem~\ref{th:iff_controllability}, the rank conditions characterizing controllability are independent of the exponent $r\in [1,+\infty)$. Consequently, if the system is controllable in $L^{r_1}$ for some $r_1\in [1,+\infty)$, then the same rank conditions are satisfied for any other exponent $r_2\in [1,+\infty)$, which implies controllability in $L^{r_2}$. Therefore, controllability in the spaces $L^r$, $r\in [1,+\infty)$, is equivalent.
\end{remark}

\begin{remark}
\label{rem:translation}
By Proposition~\ref{prop_con}, it suffices to show that conditions \ref{item:cond-rank-H} and \ref{item:cond-rank-AN} from Theorem~\ref{th:iff_controllability} are necessary conditions for exact controllability in \(\mathsf X^r\) for \(r \in (1, +\infty)\). In addition, it suffices to prove this result under the additional assumption that System~\eqref{main_system} is exponentially stable in the absence of control.

Indeed, take \(r \in (1, +\infty)\) and consider System~\eqref{main_system}. Since \((\Upsilon_r(t))_{t \geq 0}\) is a \(C^0\)-semigroup on \(\mathsf X^r\), there exist constants \(C \geq 1\) and \(\omega \in \mathbb R\) such that \(\norm{\Upsilon_r(t)}_{\mathcal L(\mathsf X^r)} \leq C e^{\omega t}\) for every \(t \geq 0\). If \(\omega < 0\), then \eqref{main_system} is already exponentially stable, so we only consider in what follows the case \(\omega \geq 0\). By \eqref{eq:solution}, any solution \(x\) of \eqref{main_system} with some initial condition \(x_0 \in \mathsf X^r\) and with null control satisfies \(\norm{x_t}_{\mathsf X^r} \leq C e^{\omega t} \norm{x_0}_{\mathsf X^r}\) for every \(t \geq 0\).

Let \(\lambda = \omega + 1\) and consider the system
\begin{equation}
\label{eq:system_translation}
y(t) = \sum_{j=1}^N e^{-\lambda \Lambda_j} A_j y(t - \Lambda_j) + B v(t),
\end{equation}
which corresponds to \eqref{main_system} with \(A_1, \dotsc, A_N\) replaced by \(e^{-\lambda \Lambda_1} A_1, \dotsc, e^{-\lambda \Lambda_N} A_N\). Clearly, \(x\) is the solution of \eqref{main_system} with some initial condition \(x_0\) and some control \(u\) if and only if the function \(y\colon t \mapsto e^{-\lambda t} x(t)\) is the solution of \eqref{eq:system_translation} with initial condition \(y_0\colon \theta \mapsto e^{-\lambda \theta} x_0(\theta)\) and control \(v\colon t \mapsto e^{-\lambda t} u(t)\). This one-to-one correspondence of trajectories allows one to establish that the exact controllability of \eqref{eq:system_translation} in \(\mathsf X^r\) is equivalent to the exact controllability of \eqref{main_system} in \(\mathsf X^r\). Moreover, for solutions with null control, we have
\begin{align*}
\norm{y_t}_{\mathsf X^r}^r & = \int_{-\Lambda_N}^0 e^{-r\lambda (t + \theta)} \abs*{x(t + \theta)}^r \, d\theta \leq e^{r \lambda \Lambda_N} e^{-r\lambda t} \norm{x_t}_{\mathsf X^r}^r \\
& \leq C^r e^{r \lambda \Lambda_N} e^{-r t} \norm{x_0}_{\mathsf X^r}^r = C^r e^{r \lambda \Lambda_N} e^{-r t} \int_{-\Lambda_N}^0 e^{r \lambda \theta} \abs{y_0(\theta)}^r \, d\theta \\
& \leq C^r e^{r \lambda \Lambda_N} e^{-r t} \norm{y_0}_{\mathsf X^r}^r,
\end{align*}
so that the semigroup \((\Upsilon^y_r(t))_{t \geq 0}\) corresponding to \eqref{eq:system_translation} in the sense of Proposition~\ref{prop:exist-unique-formula} satisfies \(\norm{\Upsilon^y_r(t)}_{\mathcal L(\mathsf X^r)} \leq C e^{\lambda \Lambda_N} e^{-t}\) for every \(t \geq 0\), and thus \eqref{eq:system_translation} is exponentially stable.

Finally, the characteristic matrix \(H_\lambda\) of \eqref{eq:system_translation} is related to the characteristic matrix \(H\) of \eqref{main_system} by \(H_\lambda(p) = H(p + \lambda)\) for all \(p \in \mathbb C\), and in particular \(H_\lambda(\mathbb C) = H(\mathbb C)\). Since in addition \(\rank \left[e^{-\lambda \Lambda_N}A_N, B\right] = \rank \left[A_N, B\right]\), we deduce that conditions \ref{item:cond-rank-H} and \ref{item:cond-rank-AN} from Theorem~\ref{th:iff_controllability} are satisfied for \eqref{main_system} if and only if the corresponding conditions are satisfied for \eqref{eq:system_translation}.
\end{remark}

\subsection{The Hautus observability test}

In order to prove Theorem~\ref{th:iff_controllability} in the case \(r \in (1, +\infty)\), we will exploit duality and a Hautus test for observability. For that purpose, let us first detail the controllability--observability duality in the context of System~\eqref{main_system}. It follows from Proposition~\ref{prop:exist-unique-formula} (see also \cite[Proposition~4.5]{chitour:hal-03827918} or \cite[Proposition~2.7]{Chitour2020Approximate}) that exact controllability in \(\mathsf{X}^r\) in time \(T\) is equivalent to \(\Ran E_r(T) = \mathsf{X}^r\). The closed range theorem yields that exact controllability of \eqref{main_system} in \(\mathsf{X}^r\) in time \(T\) is equivalent to the existence of a constant \(c_T > 0\) such that
\[
\norm{E_r(T)^\ast \varphi}_{L^{r'}((0, T), \mathbb C^m)} \geq c_T \norm{\varphi}_{\mathsf X^{r'}}, \qquad \text{ for all } \varphi \in \mathsf X^{r'},
\]
and, by density, it suffices to verify the above inequality for \(\varphi \in D(F_r^\ast)\). Using \eqref{eq:equality-Et-convo}, one computes that, for \(\varphi \in D(F_r^\ast)\), we have
\[
(E_r(T)^\ast \varphi)(t) = G_r^\ast \Upsilon_r(T - t)^\ast \varphi, \qquad \text{ for a.e.\ } t \in (0, T),
\]
so that exact controllability of \eqref{main_system} in \(\mathsf{X}^r\) in time \(T\) is equivalent to the existence of a constant \(c_T > 0\) such that
\begin{equation}
\label{eq:observability-inequality}
\norm{G_r^\ast \Upsilon_r(\cdot)^\ast \varphi}_{L^{r'}((0, T), \mathbb C^m)} \geq c_T \norm{\varphi}_{\mathsf X^{r'}}, \qquad \text{ for all } \varphi \in D(F_r^\ast).
\end{equation}

The inequality in \eqref{eq:observability-inequality} is known as an \emph{observability inequality}, since it is related to the exact observability property for the abstract system
\[
\left\{
\begin{aligned}
z^\prime(t) & = F_r^\ast z(t), \\
y(t) & = G_r^\ast z(t),
\end{aligned} \qquad t \geq 0,
\right.
\]
where \(z\) takes values in \(\mathsf X^{r'}\) and \(y\) is the observation, taking values in \(\mathbb C^m\). We refer to \cite[Chapter~6]{Tucsnak_Weiss} for more precise definitions of observability in an abstract framework for Hilbert spaces and their relations to observability inequalities.

In \cite{Russell_Weiss94} (see also \cite[Section~6.5]{Tucsnak_Weiss}), an infinite-dimensional Hautus-type condition is provided as a necessary condition for infinite-time exact observability of systems of the form \(z^\prime(t) = A z(t)\), \(y(t) = C z(t)\), for suitable operators \(A\) and \(C\) in Hilbert spaces and with observations in \(L^2\). It turns out that the arguments provided for the proofs in \cite{Russell_Weiss94} and \cite[Section~6.5]{Tucsnak_Weiss} can be easily extended to a Banach setting with observations in \(L^r\). More precisely, we have the following result.

\begin{proposition}[\(L^q\)-Hautus estimate]
\label{prop:Lr_Hautus}
Let \(q \in [1, +\infty)\), \(\mathsf X\) and \(\mathsf Y\) be Banach spaces, \(A\colon D(A) \subset \mathsf X \to \mathsf X\) be the generator of an exponentially
stable \(C_0\)-semigroup \((T(t))_{t \geq 0}\) on \(\mathsf X\), \(D(A)\) be endowed with its graph norm, and $C \in \mathcal L(D(A), \mathsf Y)$. Assume that there exist positive constants \(k\) and \(K\) such that
\begin{equation}
\label{eq:admissible-observable}
k \norm{x}_{\mathsf X}^q \leq \int_0^{+\infty} \norm{C T(t) x}_{\mathsf Y}^q \, dt \leq K \norm{x}_{\mathsf X}^q, \qquad \text{ for all } x \in D(A).
\end{equation}
Then there exists \(m_q>0\) such that, for every \(p\in\mathbb C\) with \(\Re p < 0\), we have 
\begin{equation}
\label{eq:Lr_Hautus}
\frac{1}{\abs{\Re p}^q} \norm{(p I - A) x}_{\mathsf X}^q + \frac{1}{\abs{\Re p}} \norm{C x}_{\mathsf Y}^q \geq m_q \norm{x}_{\mathsf X}^q, \qquad \text{ for every } x\in D(A).
\end{equation}
\end{proposition}

\begin{remark}
\label{remk:infinite-finite}
When the second inequality in \eqref{eq:admissible-observable} is satisfied, the operator \(C\) is said to be an \emph{infinite-time \(L^q\)-admissible observation operator} for \(A\). For exponentially stable semigroups, infinite-time admissibility is equivalent to finite-time admissibility, defined as the same property with the integral over \([0, +\infty)\) being replaced by an integral over \([0, \tau]\) for some/any \(\tau > 0\) (see, e.g., \cite[Propositions~4.3.2 and 4.3.3 and Remark~4.3.5]{Tucsnak_Weiss} for a proof of this fact in the Hilbert setting and with \(q = 2\), which readily extends to the more general setting considered here). The first inequality in \eqref{eq:admissible-observable} is an infinite-time observability inequality and, when it is satisfied, the pair of operators \((A, C)\) is said to be \emph{exactly observable in infinite time}. Once again, for exponentially stable semigroups, infinite-time exact observability is equivalent to exact observability in some finite time \(\tau > 0\) (see, e.g., \cite[Proposition~6.5.2]{Tucsnak_Weiss} for a proof of this fact in the Hilbert setting and with \(q = 2\)).
\end{remark}

For the sake of completeness, and in order to highlight the adaptations required to go from the Hilbert to the Banach setting and from \(q = 2\) to any \(q \in [1, +\infty)\), we provide below a detailed proof of Proposition~\ref{prop:Lr_Hautus}.

\begin{proof}[Proof of Proposition~\ref{prop:Lr_Hautus}]
Let \(\Psi \in \mathcal L(D(A), L^q((0, +\infty), \mathsf Y))\) be the operator defined by \(\Psi x = C T(\cdot ) x\). Thanks to the second inequality in \eqref{eq:admissible-observable}, \(\Psi\) is well-defined and it can be extended to an operator (still denoted by \(\Psi\)) in \(\mathcal L(\mathsf X, L^q((0, +\infty), \mathsf Y))\).

Let \(x \in D(A^2)\) and \(p\in\mathbb C\) with \(\Re p<0\). Setting $f = (A - p I) x$, the variation-of-constants formula gives
\[
T(t)x
=
e^{pt}x
+
\int_0^t e^{p(t-\sigma)}T(\sigma)f\,d\sigma, \qquad \text{ for all } t \geq 0.
\]
Applying \(C\) to the above identity, we obtain that $(\Psi x)(t) = e^{pt} C x + (e_p * \Psi f)(t)$ for every \(t \geq 0\), where \(*\) denotes the convolution product, \(\Psi f\) is extended by \(0\) for negative times to an element of \(L^q(\mathbb R, \mathsf Y)\), and \(e_p\) is the truncated exponential defined by \(e_p(t) = e^{p t}\) for \(t \geq 0\) and \(e_p(t) = 0\) for \(t < 0\). Young's inequality for the convolution product gives
\[
\|\Psi x\|_{L^q((0,+\infty), \mathsf Y)}
\le
\|e_p\|_{L^q((0,+\infty),\mathbb C)}\|Cx\|_{\mathsf Y}
+
\|e_p\|_{L^1((0,+\infty),\mathbb{C})}\|\Psi f\|_{L^q((0,+\infty), \mathsf Y)}.
\]
Since $\|e_p\|_{L^q((0,+\infty),\mathbb{C})}^q=\frac{1}{q|\Re p|}$ and $
\|e_p\|_{L^1((0,+\infty),\mathbb{C})}=\frac{1}{|\Re p|}$, we get
\[
\|\Psi x\|_{L^q((0,+\infty),\mathsf Y)}
\le
\frac{1}{(q|\Re p|)^{1/q}}\|Cx\|_{\mathsf Y}
+
\frac{\|\Psi\|_{\mathcal L(\mathsf X, L^q((0, +\infty), \mathsf Y))}}{|\Re p|}\|(A-pI)x\|_{\mathsf X}.
\]
Using the inequality \((a+b)^q\le 2^{q-1}(a^q+b^q)\) with $a,b>0$, we deduce that there exists a constant \(C_q > 0\) depending only on \(A\), \(C\), and \(q\) such that
\[
\|\Psi x\|_{L^q((0,+\infty), \mathsf Y)}^q
\le
C_q
\left(
\frac{1}{|\Re p|}\|Cx\|_{\mathsf Y}^q
+
\frac{1}{|\Re p|^q}\|(pI-A)x\|_{\mathsf X}^q
\right).
\]
By density of \(D(A^2)\) in \(D(A)\) for the graph norm, the above estimate extends to every \(x \in D(A)\). Combining with the first inequality of \eqref{eq:admissible-observable}, we obtain \eqref{eq:Lr_Hautus} with \(m_q = \frac{k}{C_q}\).
\end{proof}

Note that \eqref{eq:Lr_Hautus} was proved only in the case \(\Re p < 0\). When \(\Re p > 0\), classical resolvent estimates yield a stronger inequality. More precisely, we have the following result, which is an immediate corollary of \cite[Chapter~II, Theorem~1.10]{Engel2001}.

\begin{proposition}
\label{prop:spectral-bound}
Let \(\mathsf X\) be a Banach space, \(A\colon D(A) \subset \mathsf X \to \mathsf X\) be the generator of a \(C_0\)-semigroup \((T(t))_{t \geq 0}\) on \(\mathsf X\), and \(M \geq 1\) and \(\omega \in \mathbb R\) be such that \(\norm{T(t)}_{\mathcal L(\mathsf X)} \leq M e^{\omega t}\) for every \(t \geq 0\). Then, for every \(x \in D(A)\) and \(p \in \mathbb C\) with \(\Re p > \omega\), we have
\begin{equation}
\label{cplus}
\frac{1}{\Re p - \omega} \norm{(p I - A) x}_{\mathsf X} \geq \frac{1}{M} \norm{x}_{\mathsf X}.
\end{equation}
\end{proposition}

\subsection{Characterization of exact controllability via a rank condition}

To prove Theorem~\ref{th:iff_controllability}, we shall apply Propositions~\ref{prop:Lr_Hautus} and \ref{prop:spectral-bound} with \(A = F_r^\ast\) and \(p\) replaced by its complex conjugate \(\overline p\). To complete the proof of the theorem, we will thus need to relate the term \(\norm{(\overline p I - F_r^\ast)y}_{\mathsf X^{r'}}^{r'}\) from \eqref{eq:Lr_Hautus} and \eqref{cplus} with the characteristic matrix \(H(p)\) of System~\eqref{main_system}. Establishing such a relation is the goal of the next result.

\begin{lemma}
\label{lem:key_approximate_eigenvector}
Let \(r \in (1, +\infty)\). There exists a constant $C > 0$ such that, for every $p\in\mathbb C$ and $v\in\mathbb C^d$, there exists $z=z_{p,v}\in D(F_r^*)$
satisfying
\(
z(0^-)=v
\)
and
\begin{equation}
\label{lepointcle}
\|(\overline pI-F_r^*)z\|_{\mathsf{X}^{r'}}
\leq
C e^{\Lambda_N|\Re p|}
\|H(p)^*v\|_{\mathbb C^d}.
\end{equation}
In addition, we have
\begin{equation}
\label{eq:lower-bound-z}
\norm{z}_{\mathsf X^{r'}}^{r'} \geq J(\Re p) \norm{v}_{\mathbb C^d}^{r'},
\end{equation}
where \(J \colon \mathbb R \to \mathbb R_+^\ast\) is an increasing function only depending on \(r\), \(N\), and \(\Lambda_1\).
\end{lemma}

\begin{proof}
Consider the intervals \(I_j\), \(j \in \{0, \dotsc, N-1\}\) defined in Section~\ref{sec:Fr}. Take \(p \in \mathbb C\) and \(v \in \mathbb C^d\) and, for \(j \in \{0, \dotsc, N\}\), define \(C_j = I_d - \sum_{\ell = 1}^j e^{-\overline p \Lambda_\ell} A_\ell^\ast\), with the convention that the empty sum is the zero matrix. Note that, by construction, \(C_0 = I_d\) and \(C_N = H(p)^\ast\).

We define \(y \in \mathsf X^{r'}\) by setting \(y(t) = e^{-\overline p t} C_j v\) for every \(t \in I_j\) and \(j \in \{0, \dotsc, N-1\}\). Clearly, \(y_{| I_j} \in W^{1,r'}(I_j, \mathbb C^d)\) for every \(j \in \{0, \dotsc, N-1\}\). In addition, \(y(0^-) = v\) and, for \(j \in \{1, \dotsc, N-1\}\), we have
\[
y(-\Lambda_j^+) - y(-\Lambda_j^-) = e^{\overline p \Lambda_j} (C_{j-1} - C_j) v = A_j^\ast v = A_j^\ast y(0^-).
\]
By \eqref{DFe}, the only property missing in order to ensure that $y\in D(F^*_r)$ is that $y(-\Lambda_N^+) = A_N^* v$. Nothing in our construction ensures that this property is true, since we have  
$y(-\Lambda_N^+)= e^{\overline p\Lambda_N}C_{N-1}v.$

However, we can modify $y$ slightly on $I_{N-1}$, so that it verifies the above-mentioned condition. Take $\chi \in C^1([-\Lambda_N,0], [0, 1])$ such that $\chi(-\Lambda_N)=1$ and $\chi=0$ on $[-\Lambda_{N-1}-\varepsilon,0]$, where $\varepsilon>0$ is chosen small enough such that $-\Lambda_{N-1}-\varepsilon>-\Lambda_N$. Now, consider \(z \in \mathsf X^{r^\prime}\) defined by $z(t)=y(t)-\chi(t)e^{-\overline p\,t}C_Nv$ for \(t \in (-\Lambda_N, 0)\). Clearly, \(z_{|I_j} \in W^{1,r'}(I_j,\mathbb C^d)\) for every $j \in \{0, \dotsc, N-1\}$. Moreover, since $\chi=0$ on $[-\Lambda_{N-1}-\varepsilon,0]$, we have $z(0^-)=v$ and $z(-\Lambda_j^+)-z(-\Lambda_j^-) = A_j^*z(0^-)$ for $j \in \{1, \dotsc,N-1\}$.
Since $\chi(-\Lambda_N)=1$, we have
\[
z(-\Lambda_N^+) = y(-\Lambda_N^+) - e^{\overline p\Lambda_N} C_N v  = e^{\overline p\Lambda_N} (C_{N-1} - C_N) v = A_N^*v.
\]
We deduce that $z\in D(F^*_r)$.

On the intervals $I_j$ for $j \in \{0, \dotsc, N-2\}$, we have $\overline p z(t) + z'(t) = 0$, whereas, on the last interval $I_{N-1}$, an explicit computation shows that 
\[
\overline pz(t)+z'(t) = -\chi'(t) e^{-\overline {p} t} C_N v = -\chi'(t) e^{-\overline {p} t} H(p)^* v.
\]
Hence
\[\|(\overline pI-F^*_r)z\|_{\mathsf{X}^{r'}}=\|e^{-\overline {p} \cdot}\chi'(\cdot) H(p)^* v\|_{L^{r'}((-\Lambda_N,-\Lambda_{N-1}), \mathbb C^d)},\]
yielding \eqref{lepointcle} with a constant \(C\) only depending on \(r\), \(\Lambda_N - \Lambda_{N - 1}\), and on \(\norm{\chi'}_{L^{r'}((-\Lambda_N,-\Lambda_{N-1}), \mathbb R)}\). Since \(\chi\) can be constructed in such a way that the norm \(\norm{\chi'}_{L^{r'}((-\Lambda_N,-\Lambda_{N-1}), \mathbb R)}\) only depends on \(r\) and \(\Lambda_N - \Lambda_{N - 1}\), we deduce that \(C\) can be chosen to only depend on \(r\) and \(\Lambda_N - \Lambda_{N - 1}\).

To obtain the last part of the statement, let \(\delta \in (0, \Lambda_1]\) be such that \(\chi\) vanishes on \([-\delta, 0]\). Note that one can take \(\delta = \Lambda_1\) if \(N \geq 2\) and \(\delta = \varepsilon\) if \(N = 1\). We then have
\[
\norm{z}_{\mathsf X^{r'}}^{r'} \geq \int_{-\delta}^0 \norm{y(t)}_{\mathbb C^d}^{r'} \, dt = \int_{-\delta}^0 e^{-r' t \Re p} \, dt \norm{v}_{\mathbb C^d}^{r'}, 
\]
yielding the conclusion with \(J(\xi) = \int_{-\delta}^0 e^{-r' t \xi} \, dt\).
\end{proof}

We are now in a position to prove Theorem~\ref{th:iff_controllability}.

\begin{proof}[Proof of Theorem~\ref{th:iff_controllability}]
Thanks to Remark~\ref{rem:translation}, it suffices to show that conditions \ref{item:cond-rank-H} and \ref{item:cond-rank-AN} from the statement of the theorem are necessary for exact controllability in \(\mathsf X^r\) for \(r \in (1, +\infty)\), and under the additional assumption that System~\eqref{main_system} is exponentially stable in the absence of control.

Take \(r \in (1, +\infty)\) and assume thus that System~\eqref{main_system} is exponentially stable in the absence of control and that it is exactly controllable in \(\mathsf X^r\). Let \(M \geq 1\) and \(\omega > 0\) be such that \(\norm{\Upsilon_r(t)}_{\mathcal L(\mathsf X^r)} \leq M e^{-\omega t}\) for every \(t \geq 0\). Note that, since \eqref{main_system} is exactly controllable in \(\mathsf X^r\), it is also approximately controllable in this space (we refer to \cite{chitour:hal-03827918} for the definition of approximate controllability), and thus, by \cite[Theorem~3.6]{chitour:hal-03827918}, condition \ref{item:cond-rank-AN} from the statement is satisfied. In addition, by \cite[Proposition~3.12]{chitour:hal-03827918}, condition \ref{item:cond-rank-H} is equivalent to the existence of $\alpha>0$ such that
\begin{equation}
\label{coer}
\inf \left\{\norm{g^T H(p)}_{\mathbb{C}^d} + \norm{g^T B}_{\mathbb{C}^m} \suchthat g \in \mathbb{C}^d,\, \norm{g}_{\mathbb{C}^d} = 1 \right\} \geq \alpha, \qquad \text{ for all } p \in \mathbb C.
\end{equation}

We split the complex plane \(\mathbb C\) into four regions \(\mathbb C = \mathcal S_1 \cup \mathcal S_2 \cup \mathcal S_3 \cup \mathcal S_4\), where
\begin{align*}
\mathcal S_1 & = \left\{p \in \mathbb C \suchthat \Re p < \gamma_-\right\}, &
\mathcal S_2 & = \left\{p \in \mathbb C \suchthat \gamma_- \leq \Re p < -\frac{\omega}{2}\right\}, \\
\mathcal S_3 & = \left\{p \in \mathbb C \suchthat -\frac{\omega}{2} \leq \Re p \leq \gamma_+ \right\}, &
\mathcal S_4 & = \left\{p \in \mathbb C \suchthat \Re p > \gamma_+ \right\},
\end{align*}
with \(\gamma_-\) and \(\gamma_+\) to be fixed later. We prove \eqref{coer} separately in each of these regions.

Let us first prove \eqref{coer} in \(\mathcal S_4\) for a suitable choice of \(\gamma_+\). Since \(\Lambda_j > 0\) for every \(j \in \{1, \dotsc, N\}\), there exists \(\gamma_+ > 0\) such that \(\sum_{j=1}^N e^{-\Re p \Lambda_j} \norm{A_j}_{\mathbb C^{d \times d}} < \frac{1}{2}\) for every \(p \in \mathbb C\) with \(\Re p > \gamma_+\). Hence, for every \(p \in \mathcal S_4\) and $g\in\mathbb C^d$ with $\|g\|_{\mathbb{C}^d}=1$, we have
\[
\|g^T H(p)\|_{\mathbb{C}^d} = \norm*{g^T-\sum_{j=1}^N e^{-p\Lambda_j} g^T A_j}_{\mathbb{C}^d} 
\ge
1 - \sum_{j=1}^N e^{-\Re p\Lambda_j} \norm{A_j}_{\mathbb C^{d \times d}} \ge \frac{1}{2},
\]
yielding \eqref{coer} for \(p \in \mathcal S_4\).

We now establish \eqref{coer} in \(\mathcal S_1\) for a suitable choice of \(\gamma_-\). We first remark that, since $\rank [A_N, B] = d$, the map $g \mapsto \big(g^TA_N,\ g^TB\big)$ is injective, and thus the compactness of the unit sphere of $\mathbb C^d$ yields the existence of $\kappa>0$ such that
\begin{equation}\label{eq:ANB_coercive}
\|g^TA_N\|_{\mathbb{C}^d}+\|g^TB\|_{\mathbb{C}^m}\ge \kappa,
\quad \text{ for all } g \in \mathbb C^d \text{ with } \|g\|_{\mathbb{C}^d}=1.
\end{equation}
We write
\[
\begin{aligned}
g^TH(p)
=
g^T-\sum_{j=1}^N e^{-p\Lambda_j}g^TA_j 
=
-e^{-p\Lambda_N}g^T\left(A_N+R(p)\right),
\end{aligned}
\]
with $R(p)\coloneqq 
\sum_{j=1}^{N-1}e^{p(\Lambda_N-\Lambda_j)}A_j
-
e^{p\Lambda_N}I_d$. Since $\Lambda_N-\Lambda_j>0$ for every $j \in \{1, \dotsc, N-1\}$, there exists $\gamma_- < -\frac{\omega}{2}$ such that, for every $p \in \mathbb C$ with $\Re p < \gamma_-$, we have $\norm{R(p)}_{\mathbb C^{d \times d}} \le \frac{\kappa}{2}$. In addition, since \(\gamma_- < 0\), we have \(e^{-\Re p \Lambda_N} \geq 1\) whenever \(\Re p < \gamma_-\). Thus, for every \(p \in \mathcal S_1\) and \(g \in \mathbb C^d\) with $\|g\|_{\mathbb{C}^d}=1$, we have
\[
\begin{aligned}
\|g^TH(p)\|_{\mathbb{C}^d}+\|g^TB\|_{\mathbb{C}^m}
&=
e^{-\Re p\Lambda_N}\|g^T(A_N+R(p))\|_{\mathbb{C}^d}+\|g^TB\|_{\mathbb{C}^m} \\
&\ge \|g^TA_N\|_{\mathbb{C}^d}+\|g^TB\|_{\mathbb{C}^m}-\norm{R(p)}_{\mathbb C^{d \times d}} \geq
\frac{\kappa}{2},
\end{aligned}
\]
yielding \eqref{coer} for \(p \in \mathcal S_1\).

To prove \eqref{coer} in \(\mathcal S_2\), recall that, by Remark~\ref{remk:adjoint-semigroup}, \(F_r^\ast\) is the generator of the adjoint semigroup \((\Upsilon_r(t)^\ast)_{t \geq 0}\) in \(\mathsf X^{r'}\) and, since \(\mathsf X^r\) is reflexive, we have that \(\norm{\Upsilon_r(t)}_{\mathcal L(\mathsf X^r)} = \norm{\Upsilon_r(t)^\ast}_{\mathcal L(\mathsf X^{r^\prime})}\) for every \(t \geq 0\), so that \((\Upsilon_r(t)^\ast)_{t \geq 0}\) is also exponentially stable. By Theorem~\ref{th_system_abstract}, \(G_r\) is an admissible control operator for \(F_r\). Using duality and Remark~\ref{remk:infinite-finite}, we deduce that \(G_r^\ast\) is an infinite-time \(L^{r'}\)-admissible observation operator for \(F_r^\ast\). The exact controllability assumption on System~\eqref{main_system} yields \eqref{eq:observability-inequality}, which implies that \((F_r^\ast, G_r^\ast)\) is exactly observable in infinite time. Hence, all assumptions of Proposition~\ref{prop:Lr_Hautus} are satisfied in the case \(q=r'\), \(A=F_r^*\), and \(C=G_r^*\), and applying \eqref{eq:Lr_Hautus} with \(p\) replaced by \(\overline p\) and using the explicit expression of \(G_r^\ast\) from \eqref{eq:def-Gr}, we deduce that there exists $\mu > 0$ such that, for every $y \in D(F_r^*)$ and $p \in \mathbb C$ with $\Re p < 0$, we have
\begin{equation}
\label{eq:Hautus_Yamamoto_explicit}
\frac{1}{\abs{\Re p}^{r'}} \norm{(\overline p I - F_r^\ast) y}_{\mathsf X^{r'}}^{r'} + \frac{1}{\abs{\Re p}} \norm*{B^* y(0^-)}_{\mathbb C^m}^{r'} \geq \mu \norm*{y}_{\mathsf X^{r'}}^{r'}.
\end{equation}
Take \(p \in \mathcal S_2\) and \(g \in \mathbb C^d\) with \(\norm{g}_{\mathbb C^d} = 1\). Let \(y \in D(F_r^\ast)\) be the element obtained by applying Lemma~\ref{lem:key_approximate_eigenvector} with \(v = g\). Then, combining \eqref{lepointcle}, \eqref{eq:lower-bound-z}, and \eqref{eq:Hautus_Yamamoto_explicit}, we obtain
\[
\frac{C^{r'} e^{r' \Lambda_N \abs{\Re p}}}{\abs{\Re p}^{r'}} \norm{H(p)^\ast g}_{\mathbb C^d}^{r'} + \frac{1}{\abs{\Re p}} \norm*{B^* g}_{\mathbb C^m}^{r'}
 \geq \mu J(\gamma_-),
\]
yielding \eqref{coer} in \(\mathcal S_2\) thanks to the upper and lower bounds of \(\Re p\) in this set.

Finally, to prove \eqref{coer} in \(\mathcal S_3\), we apply Proposition~\ref{prop:spectral-bound} with \(A = F_r^\ast\) and \(p\) replaced by \(\overline p\) to obtain that, for every \(y \in D(F_r^\ast)\) and \(p \in \mathbb C\) with \(\Re p > -\omega\), we have
\begin{equation}
\label{eq:spectral-lower-bound}
\frac{1}{\Re p + \omega} \norm{(\overline p I - F_r^\ast) y}_{\mathsf X^{r'}} \geq \frac{1}{M} \norm{y}_{\mathsf X^{r'}}.
\end{equation}
Take \(p \in \mathcal S_3\) and \(g \in \mathbb C^d\) with \(\norm{g}_{\mathbb C^d} = 1\) and, as before, let \(y \in D(F_r^\ast)\) be the element obtained by applying Lemma~\ref{lem:key_approximate_eigenvector} with \(v = g\). Combining now \eqref{lepointcle}, \eqref{eq:lower-bound-z}, and \eqref{eq:spectral-lower-bound}, we obtain
\[
\frac{C e^{\Lambda_N \abs{\Re p}}}{\Re p + \omega} \norm{H(p)^\ast g}_{\mathbb C^d} \geq \frac{J(-\frac{\omega}{2})^{1/r'}}{M},
\]
which yields \eqref{coer} in \(\mathcal S_3\) using the upper and lower bounds of \(\Re p\) in this set.
\end{proof}

\begin{remark}
It is worth emphasizing that the necessity argument in the proof of Theorem~\ref{th:iff_controllability} for $r\in(1,+\infty)$ does not extend to the endpoint case $r=1$, since one cannot apply Proposition~\ref{prop:Lr_Hautus} with $q = r' = +\infty$. As a result, the necessity of the two rank conditions appearing in Theorem~\ref{th:iff_controllability} for $L^1$-exact controllability currently relies solely on the proof provided in \cite{chitour:hal-03827918}.
\end{remark}

\section{Rapid stabilization}
\label{sec:stabilization}

Following the framework introduced in \cite{Komornik} and further developed in \cite{Urquiza2005Rapid, Vest2013Rapid}, we now investigate the rapid stabilization of System~\eqref{main_system} in the $L^2$ setting through its abstract formulation \eqref{eq:abstract}. More precisely, we wish to understand whether, for every
\(\omega>0\), one can find
a linear state feedback law
\(u(t)=K_\omega z(t)\) such that
the corresponding closed-loop system is well posed
in \(\mathsf X^2\) and
exponentially stable with an exponential rate of \(-\omega\).
Here, the feedback law is understood in the sense of
\cite{Vest2013Rapid}, as recalled below.

Let us start by recalling the framework of \cite{Komornik, Urquiza2005Rapid, Vest2013Rapid}. Let \(\mathsf H\) and \(\mathsf U\) be Hilbert spaces\footnote{Throughout this section, we identify \(\mathsf H\) and \(\mathsf U\) with their duals.}.
We consider the abstract control system
\begin{equation}
\label{eq:abstract_system}
\left\{
\begin{aligned}
x'(t) & = \mathcal{A}x(t) + \mathcal{B}u(t), \\
x(0) & = x_0,
\end{aligned}
\right.
\end{equation}
where \(\mathcal{A}\colon D(\mathcal{A})\subset \mathsf H\to \mathsf H\) is the generator of a semigroup $(e^{\mathcal A t})_{t \ge 0}$, \(\mathcal{B}\) is a control operator, and \(u\) denotes the control input, which takes values in \(\mathsf U\).

The rapid stabilization result from \cite{Vest2013Rapid} relies on the following assumptions.

\begin{hypothesis}
\label{hyp_Komornik}
\begin{enumerate}[label=(H\arabic*)]
\item
\label{hyp_h1}
The operator $\mathcal{A}$ generates a $C_0$-group on $\mathsf H$;

\item
\label{hyp_h2}
$\mathcal{B} \in \mathcal{L}(\mathsf U, D(\mathcal{A}^*)')$;

\item
\label{hyp_h3}
\(\mathcal B^\ast\) is an admissible observation operator for \(\mathcal A^\ast\), that is, for every $T'>0$, there exists a constant $c' > 0$ such that
\begin{equation}
\label{eq_label_h3}
\|\mathcal{B}^* e^{\mathcal A^\ast \cdot} y \|_{L^2((0,T'), \mathsf U)}
\le c' \|y\|_{\mathsf H},
\quad \forall y \in D(\mathcal{A}^*);
\end{equation}

\item
\label{hyp_h4}
There exists $T>0$ such that System~\eqref{eq:abstract_system} is exactly controllable in time $T$.
\end{enumerate}
\end{hypothesis}

For every $\omega>0$, set $T_\omega \coloneqq T + \frac{1}{2\omega}$
and consider the weight function $e_\omega \colon [0,T_\omega] \to \mathbb{R}_+$ defined by
\[
e_\omega(s) \coloneqq
\begin{cases}
e^{-2\omega s}, & \text{if } 0 \leq s \leq T, \\
2\omega e^{-2\omega T}(T_\omega - s), & \text{if } T \leq s \leq T_\omega.
\end{cases}
\]
Based on this weight, we define the operator $\Lambda_\omega \colon \mathsf H \to \mathsf H$ by setting, for \(x\) and \(y\) in \(\mathsf H\),
\[
\langle \Lambda_\omega x, y \rangle_{\mathsf H}
=
\int_0^{T_\omega} e_\omega(s)\,
\langle \mathcal{B}^* e^{-s\mathcal{A}^*} x,\; \mathcal{B}^* e^{-s\mathcal{A}^*} y \rangle_{\mathsf U}\,ds.
\]
The observation terms in this formula are understood through their admissible extensions, which are available also for negative times since $\mathcal A$ generates a group.
The operator $\Lambda_\omega$ is self-adjoint, positive definite, and boundedly invertible, with inverse denoted by $\Lambda_\omega^{-1}$.
The corresponding feedback operator is then defined by
\begin{equation*}
K_\omega \coloneqq -\mathcal B^*\Lambda_\omega^{-1}
\colon D_\omega\to\mathsf U,
\qquad \text{ where } D_\omega\coloneqq\Lambda_\omega D(\mathcal A^*).
\end{equation*}
In order to provide a precise sense to the feedback law \(u(t) = K_\omega x(t)\) applied to \eqref{eq:abstract_system}, we need to consider the canonical extension  $\widetilde{\mathcal A}\in\mathcal L(\mathsf H,D(\mathcal A^*)')$ of $\mathcal A$, which is an operator defined on the whole space \(\mathsf H\). This allows one to consider the closed-loop operator defined by
\[
\mathcal A_\omega x
\coloneqq\widetilde{\mathcal A}x
-\mathcal B\mathcal B^*\Lambda_\omega^{-1}x,
\qquad D(\mathcal A_\omega)=D_\omega,
\]
as well as the associated closed-loop system
\begin{equation}
\label{closed_loop_system}
x'(t)=\mathcal A_\omega x(t).
\end{equation}
The rapid stabilization result we will rely on is the following.

\begin{theorem}[{\cite[Theorems~3.2 and~3.4, Section~4]{Vest2013Rapid}}]
\label{lemma_H}
If System~\eqref{eq:abstract_system} verifies Hypotheses~\ref{hyp_h1}--\ref{hyp_h4}, then it is rapidly stabilizable in $\mathsf H$. More precisely, for every $\omega>0$, the operator $\mathcal A_\omega$, with domain $D_\omega$, generates a $C_0$-group $(S_\omega(t))_{t\in\mathbb R}$ on $\mathsf H$, and the solutions of the closed-loop system~\eqref{closed_loop_system} satisfy the following: there exists a constant $c > 0$ such that, for every initial datum $x_0 \in \mathsf H$, the corresponding mild solution \(x\) of \eqref{closed_loop_system} is given by $x(t)=S_\omega(t)x_0$ and satisfies
\begin{equation}
\|x(t)\|_{\mathsf H} \le c\, e^{-\omega t} \|x_0\|_{\mathsf H}, \quad \forall t \ge 0.
\end{equation}
\end{theorem}

For $x_0\in D_\omega$, the corresponding control is given pointwise by $u(t)=K_\omega S_\omega(t)x_0$. For general $x_0\in\mathsf H$, this formula is understood by continuous extension from $D_\omega$ to $\mathsf H$ as an $L^2_{\mathrm{loc}}([0,+\infty),\mathsf U)$-valued map. Moreover, $K_\omega\in\mathcal L(D_\omega,\mathsf U)$ when $D_\omega=D(\mathcal A_\omega)$ is endowed with its graph norm.

Our main result in this section is the following, which uses Theorem~\ref{lemma_H} to provide a rapid stabilization result for System~\eqref{main_system}.

\begin{theorem}
\label{main_result1} System~\eqref{main_system} is rapidly stabilizable in $\mathsf{X}^2$  if 
\begin{enumerate}[1)]
\item \label{assumption1_stab}$\rank \left[M,B\right]=d$ for every  $M\in \overline{H(\mathbb{C})}$,
\item \label{assumption2_stab} $\rank [A_N,B]=d$.
\end{enumerate}
\end{theorem}
\begin{proof}
Assume that conditions~\ref{assumption1_stab} and \ref{assumption2_stab} of the statement of the theorem hold. The operator \(F_2\) need not generate a \(C^0\)-group on \(\mathsf X^2\), so we first introduce a preliminary feedback before applying Theorem~\ref{lemma_H}.
Condition~\ref{assumption2_stab} implies that there exists a matrix \(K\) such that \(\widehat A_N\coloneqq A_N+BK\) is invertible. Indeed, this condition yields matrices \(C\in\mathbb C^{d\times d}\) and \(D\in\mathbb C^{m\times d}\) such that $A_NC+BD=I_d$. For every $\varepsilon>0$ except a finite number of values, set 
\(K_\varepsilon=D(C+\varepsilon I_d)^{-1}\). Then, one checks that $A_N+BK_\varepsilon = (I_d+\varepsilon A_N)(C+\varepsilon I_d)^{-1},$ and hence one chooses $K=K_{\bar\varepsilon}$ for some $\bar\varepsilon>0$ such that \(I_d+\varepsilon A_N\) is invertible.

Consider the system
\begin{equation}
\label{system_lin_l2bis}
x(t)=\sum_{j=1}^{N-1}A_jx(t-\Lambda_j)
+\widehat A_Nx(t-\Lambda_N)+Bv(t),\qquad t\ge0,
\end{equation}
with the same initial condition as System~\eqref{main_system}.
It is obtained from \eqref{main_system} by setting
\(u(t)=v(t)+K_0x(t-\Lambda_N)\).
For each fixed initial condition and each \(T>0\), this relation gives a one-to-one correspondence between the controls in \(L^2((0,T),\mathbb C^m)\) of the two systems, with the same state trajectories.
Indeed, both systems have solutions in \(L^2_{\rm loc}((-\Lambda_N,+\infty),\mathbb C^d)\) by Proposition~\ref{prop:exist-unique-formula}, so the delayed term belongs to \(L^2((0,T),\mathbb C^m)\).
Since System~\eqref{main_system} is exactly controllable by Proposition~\ref{prop_con}, so is System~\eqref{system_lin_l2bis}.

Let \(\widehat F_2\) and \(\widehat G_2\) denote the state and control operators associated with \eqref{system_lin_l2bis} through the construction of Section~\ref{sec:representation}.
We now check the four assumptions of Theorem~\ref{lemma_H} with
\(\mathcal A=\widehat F_2\), \(\mathcal B=\widehat G_2\),
\(\mathsf H=\mathsf X^2\), and \(\mathsf U=\mathbb C^m\). We first prove \ref{hyp_h1}. Let \(x\) be a solution of \eqref{system_lin_l2bis} with null control and initial condition \(x_0\in\mathsf X^2\).
Writing \(x_t=\widehat\Upsilon_2(t)x_0\), we have, for a.e.\ \(t\ge0\),
\begin{equation}
\label{eq:forward}
x(t)=\sum_{j=1}^{N-1}A_jx(t-\Lambda_j)
+\widehat A_Nx(t-\Lambda_N).
\end{equation}
Define \(y(t)=x(-t-\Lambda_N)\) for \(t\le0\). 
Since \(\widehat A_N\) is invertible, 
\eqref{eq:forward} yields, for a.e.\ \(t\le0\),
\begin{equation}
\label{eq:backward-change}
y(t)=\widehat A_N^{-1}y(t-\Lambda_N)
-\sum_{j=1}^{N-1}\widehat A_N^{-1}A_j
y(t-\Lambda_N+\Lambda_j).
\end{equation}
After reordering the delays, this is again a difference delay equation of the form \eqref{main_system}.
By Proposition~\ref{prop:exist-unique-formula}, \(y\) therefore extends uniquely to a function in \(L^2_{\rm loc}(\mathbb R,\mathbb C^d)\) satisfying \eqref{eq:backward-change} almost everywhere on \(\mathbb R\).
Reversing time again gives a unique extension of \(x\) satisfying \eqref{eq:forward} a.e.\ on \(\mathbb R\).
More precisely, let \(R\in\mathcal L(\mathsf X^2)\) be the reflection operator
\[
(Rz)(\theta)=z(-\theta-\Lambda_N),
\qquad \theta\in(-\Lambda_N,0),
\]
and let \((\widetilde\Upsilon_2(t))_{t\ge0}\) be the \(C^0\)-semigroup associated with \eqref{eq:backward-change}.
The backward evolution of the histories is given by
$\widehat\Upsilon_2(-t)
=R\widetilde\Upsilon_2(t)R$ for $t\ge0$.
Uniqueness of the forward and backward solutions shows that this defines a \(C^0\)-group extending \((\widehat\Upsilon_2(t))_{t\ge0}\).
Thus \(\widehat F_2\) satisfies \ref{hyp_h1}.
As for \ref{hyp_h2}, it follows directly from the definition of \(\widehat G_2\) while \ref{hyp_h3} follows from Theorem~\ref{th_system_abstract}, applied to System~\eqref{system_lin_l2bis}, and the duality between admissibility for control and observation operators.
Finally, \ref{hyp_h4} follows from the exact controllability of \eqref{system_lin_l2bis} established above.

Theorem~\ref{lemma_H} therefore applies to the transformed system.
For every \(\omega>0\), it provides a feedback law
\(v(t)=\widehat K_\omega x_t\) yielding a well-posed closed-loop system in \(\mathsf X^2\) with exponential decay rate \(-\omega\).
The corresponding feedback law for System~\eqref{main_system} is $u(t)=\widehat K_\omega x_t+K_0x(t-\Lambda_N)$
understood along the closed-loop trajectories.
The two closed-loop systems have the same state trajectories, and hence the same exponential decay estimate.
The conclusion follows.
\end{proof}

\begin{remark}
The rapid stabilization result presented above remains intrinsically Hilbertian, since the stabilization construction in \cite{Komornik} uses
a quadratic Gramian and the Hilbert-space identification of a space with its
dual. It does not seem immediate to achieve rapid stabilization in $\mathsf{X}^r$ for $r\neq 2$. 
\end{remark}

\section{Linear one-dimensional hyperbolic systems}
\label{sec:PDE}

The aim of this section is to show how the results of Sections~\ref{sec:controllability} and \ref{sec:stabilization} can be adapted to cover also linear one-dimensional hyperbolic systems. We start by describing the system of hyperbolic PDEs that we consider. Let $\ell$, $n$, and $m$ be positive integers with \(\ell \leq n\), $T>0$ be a positive time, and $r\in [1,+\infty)$. Fix $B \in \mathbb C^{n \times m}$ and $K = (K_{ij})_{i, j \in \{1, \dotsc, n\}} \in \mathbb C^{n \times n}$ and consider 
the set $\Omega$ defined by
\begin{equation} 
  \label{eq:Omega}
    \Omega = \{(t,x)\in \mathbb{R}^2 \mid 0<x<1\text{ and } t > 0\}.
\end{equation}
For $i \in \{1, \dotsc, n\}$, let $\lambda_i\colon[0,1]\to \mathbb{R}$ be an essentially bounded measurable function such that $ \lambda_i(x) \ge \varepsilon$ for some $\varepsilon>0$ and almost every $x \in [0,1]$. We consider the hyperbolic system
\begin{numcases}{} 
\label{eq_hyp}
\partial_t R(t,x)+\Lambda(x) \partial_x R(t,x)=0, & \((t,x)\in \Omega,\) \\
\label{eq_bord}
\begin{pmatrix}
R^+(t,0) \\
R^{-}(t,1)
\end{pmatrix}=K \begin{pmatrix}
R^+(t,1) \\
R^{-}(t,0)
\end{pmatrix}+Bu(t), & \(t \ge 0,\)
\end{numcases}
where \(R\) is the state of the system, taking values in \(\mathbb C^n\), $u(\cdot) \in \mathsf{U}^r$ denotes the control, the matrix \(\Lambda(x)\) is defined for \(x \in (0, 1)\) by
\begin{align}
\label{Def_lambda}
\Lambda (x) & = \begin{pmatrix}
     \Lambda^+(x) & 0 \\
     0 & -\Lambda^-(x)
  \end{pmatrix}
, & \text{ where } & \left\{
  \begin{aligned}
    \Lambda^+(x) &= \diag\{\lambda_1(x), \dotsc, \lambda_\ell(x)\}, \\
    \Lambda^-(x) &= \diag\{\lambda_{\ell+1}(x), \dotsc, \lambda_n(x)\}, \\
  \end{aligned}
\right. \\
\intertext{and the state \(R\) is decomposed as}
R & = \begin{pmatrix}
 R^+ \\
 R^-
 \end{pmatrix}, & \text{ where } & \left\{ \begin{aligned}
    R^+ &= (R_1, \dotsc, R_\ell)^T, \\
    R^- &= (R_{\ell+1}, \dotsc, R_n)^T.
  \end{aligned} \right.
\end{align} 

We now reduce the above hyperbolic system to a controlled difference delay system. For each $i \in \{1, \dotsc, n\}$, we define $\tau_i = \int_0^1 \frac{1}{\lambda_i(x)}\,dx$, which represents the travel time associated with the $i$-th characteristic. We set $\tau_{\max}=\max_{i\in\{1,\ldots,n\}}\tau_i$. Applying the method of characteristics, it is classical (see, for instance, \cite{CHITOUR2024_contr_hyp}) that System~\eqref{eq_hyp}--\eqref{eq_bord} can be reduced to the difference delay system
\begin{equation}
\label{syst_lin_avec_sortie_EDP}
\begin{pmatrix}
 R_1(t,0) \\
 \vdots\\
 R_\ell(t,0)\\
  R_{\ell+1}(t,1)\\
 \vdots\\
 R_n(t,1)
 \end{pmatrix}=K \begin{pmatrix}
 R_1(t-\tau_1,0) \\
 \vdots\\
 R_\ell(t-\tau_\ell,0)\\
  R_{\ell+1}(t-\tau_{\ell+1},1)\\
 \vdots\\
 R_n(t-\tau_n,1)
 \end{pmatrix}+Bu(t), \quad \text{ for a.e.\ }t \ge 0.
\end{equation}

System~\eqref{syst_lin_avec_sortie_EDP} constitutes a particular subclass of the general difference delay systems described by~\eqref{main_system}. However, as detailed in \cite{CHITOUR2024_contr_hyp}, in order to properly establish a correspondence between controllability properties of \eqref{eq_hyp}--\eqref{eq_bord} and those of \eqref{syst_lin_avec_sortie_EDP}, the latter cannot be considered in the full state space $L^r((-\tau_{\max}, 0), \mathbb{C}^n)$, but should rather be considered in the space \(\mathsf S^r\) defined by
\[
\mathsf{S}^r=\prod_{i=1}^{n} L^r((-\tau_i, 0), \mathbb{C}).
\]
Indeed, it follows from the method of characteristics that only the values of the \(i\)-th component of \eqref{syst_lin_avec_sortie_EDP} in the interval \((-\tau_i, 0)\) are needed in order to recover the \(i\)-th component \(R_i\) of \eqref{eq_hyp} in the space interval \((0, 1)\). It is worth mentioning that the controllability properties of \eqref{syst_lin_avec_sortie_EDP} in $L^r((-\tau_{\max}, 0), \mathbb{C}^n)$ and in $\mathsf{S}^r$ are not equivalent, as shown in \cite{CHITOUR2024_contr_hyp}. Therefore, the framework developed in the previous sections cannot be applied directly and must be suitably adapted.

To discuss the properties of System~\eqref{syst_lin_avec_sortie_EDP}, we rewrite it in the simpler form
\begin{equation}
\label{eq_generic_delay}
\begin{pmatrix}
 y_1(t) \\
 \vdots\\
 y_n(t)
 \end{pmatrix}=K \begin{pmatrix}
 y_1(t-\tau_1) \\
 \vdots\\
 y_n(t-\tau_n)
 \end{pmatrix}+Bu(t), \quad \text{ for a.e.\ $t \ge 0$},
\end{equation}
where \(y_i(t) = R_i(t, 0)\) for \(i \in \{1, \dotsc, \ell\}\) and \(y_i(t) = R_i(t, 1)\) for \(i \in \{\ell + 1, \dotsc, n\}\). For a solution $y$ of \eqref{eq_generic_delay} and \(t \geq 0\), we define the corresponding state $y_t$ of \eqref{eq_generic_delay} in \(\mathsf S^r\) as the element of $\mathsf{S}^r$ defined by $(y_t)_i(\theta)=y_i(t+\theta)$ for $\theta\in (-\tau_i, 0)$ and $i \in \{1, \dotsc, n\}$, and we equip $\mathsf{S}^r$ with the norm $\norm{\cdot}_{S^r}$ defined by
\[
\|y\|_{\mathsf{S}^r}
= 
\left(
\sum_{i=1}^n
\|y_i\|_{L^r((-\tau_i, 0), \mathbb C)}^r
\right)^{1/r}.
\]
With this choice of norm, the dual space of $\mathsf{S}^r$ is identified with $\mathsf{S}^{r'}$. We recall that, by \cite{CHITOUR2024_contr_hyp}, given an initial condition $y_0\in \mathsf{S}^r$ and a control input $u\in \mathsf{U}^r$, System~\eqref{eq_generic_delay} admits a unique solution taking values in $\mathsf{S}^r$. In addition, the method of characteristics yields an isomorphism between the state space $L^r((0,1),\mathbb{C}^n)$ of the hyperbolic system~\eqref{eq_hyp}--\eqref{eq_bord} and the space $\mathsf{S}^r$, mapping solutions of the hyperbolic system onto solutions of the associated delay system~\eqref{eq_generic_delay}. Consequently, exact controllability and, in the case $r=2$, exponential stabilization properties are equivalent for both systems. We refer to \cite[Propositions~3.4, 3.7, and 3.19]{CHITOUR2024_contr_hyp} for more details.

The $L^r$-exact controllability and the $L^2$ rapid stabilization problems for System~\eqref{eq_generic_delay}, arising from hyperbolic PDEs, are closely related to those of general difference delay systems~\eqref{main_system}. The only difference lies in the underlying state space in which the dynamics evolve. Accordingly, it is necessary to introduce the counterpart of the holomorphic function $H$ defined in \eqref{hqp} in the context of System~\eqref{eq_generic_delay}. We thus define\footnote{In \cite{CHITOUR2024_contr_hyp}, \(H_{\rm hyp}\) was defined instead by \(H_{\rm hyp}(p) = \diag(e^{p \tau_1},\dotsc,e^{p \tau_n}) - K\). It is easy to check that the exact controllability condition from \cite[Theorem~5.2, item 1]{CHITOUR2024_contr_hyp} using this function \(H_{\rm hyp}\) is equivalent to the exact controllability condition stated in this paper using \eqref{eq:def-H-hyp}.} \(H_{\rm hyp}\colon \mathbb C \to \mathbb C^{n \times n}\), by setting, for $p\in \mathbb{C}$,
\begin{equation}
\label{eq:def-H-hyp}
H_{\rm hyp}(p) = I_n- K \diag(e^{-p \tau_1},\dotsc,e^{-p \tau_n}).
\end{equation}

The counterpart of Proposition~\ref{prop:control-equivalence-time} was established for \eqref{eq_generic_delay} in \cite[Theorem~4.1]{CHITOUR2024_contr_hyp}, with the time \(d \Lambda_N\) in assertions \ref{item:contr-dLambdaN} and \ref{item:contr-geq-dLambdaN} of Proposition~\ref{prop:control-equivalence-time} replaced by the time \(\tau_\ast = \tau_1 + \dotsb + \tau_n\). Hence, in the sequel, we will say that \eqref{eq_generic_delay} is \emph{exactly controllable in \(\mathsf S^r\)} if it is exactly controllable in \(\mathsf S^r\) in time \(\tau_\ast\).

The counterpart of Proposition~\ref{prop_con} for \eqref{eq_generic_delay} was also established in~\cite{CHITOUR2024_contr_hyp}. More precisely, it is shown in~\cite{CHITOUR2024_contr_hyp} that System~\eqref{eq_generic_delay} is exactly controllable in \(\mathsf S^r\) if the following rank conditions hold:
\begin{enumerate}[1)]
\item \label{assumption1_wave} $\rank [M, B]=n$ for every $M \in \overline{H_{\rm hyp}(\mathbb{C})}$,
\item \label{assumption2_wave} $\rank [K, B]=n$.
\end{enumerate}
These two rank conditions are also necessary when $r=1$.

We now wish to prove that these two rank conditions are also necessary for \(r \in (1, +\infty)\) and that, in the case \(r = 2\), they imply rapid stabilization of \eqref{eq_hyp}--\eqref{eq_bord}.

To do so, we first describe the main modifications required in Section~\ref{sec:representation} in order to write \eqref{eq_generic_delay} under the abstract form \eqref{eq:abstract}. From now on, we assume that $r\in(1, +\infty)$. The essential difference with respect to the framework of Section~\ref{sec:representation} is that the components of the state are now defined on different intervals. Instead of \eqref{eq:Upsilon}, the semigroup \((\Upsilon_r^{\rm hyp}(t))_{t \geq 0}\) on \(\mathsf S^r\) corresponding to \eqref{eq_generic_delay} is now defined componentwise, for \(t \in [0, \min(\tau_1, \dotsc, \tau_n))\), \(y = (y_i)_{i \in \{1, \dotsc, n\}} \in \mathsf S^r\), \(i \in \{1, \dotsc, n\}\), and a.e.\ \(\theta \in (-\tau_i, 0)\) by
\[
(\Upsilon_r^{\rm hyp}(t) y)_i (\theta) = \begin{dcases*}
\sum_{j = 1}^n K_{i j} y_j(t + \theta - \tau_j) & if \(0 \leq t + \theta < \min(\tau_1, \dotsc, \tau_n)\), \\
y_i(t + \theta) & if \(t + \theta < 0\),
\end{dcases*}
\]
and, adapting the proof of Proposition~\ref{prop:generator-Fr}, one obtains that its generator is the operator \(F_r^{\rm hyp}\colon D(F_r^{\rm hyp}) \subset \mathsf S^r \to \mathsf S^r\) defined by
\begin{align*}
D(F_r^{\rm hyp}) & = \left\{y \in \prod_{i = 1}^n W^{1, r}((-\tau_i, 0), \mathbb C) \suchthat y(0) = K \begin{pmatrix}y_1(-\tau_1) \\ \vdots \\ y_n(-\tau_n)\end{pmatrix}\right\}, \displaybreak[0] \\
(F_r^{\rm hyp} y)_i & = y_i' \text{ on } (-\tau_i, 0).
\end{align*}
By the same integration-by-parts argument as in the proof of Proposition~\ref{prop:adjoint_F}, we obtain that the adjoint \((F_r^{\rm hyp})^\ast\colon D((F_r^{\rm hyp})^\ast) \subset \mathsf S^{r'} \to \mathsf S^{r'}\) of \(F_r^{\rm hyp}\) is the operator defined by
\begin{align*}
D\left((F^{\rm hyp}_r)^*\right) & = \Biggl\{y \in \prod_{i = 1}^n W^{1, r'}((-\tau_i, 0), \mathbb C) \mathop{}\Bigg\vert\mathop{} \\
& \hspace*{4em} y_i(-\tau_i) = \bigl(K^* y(0)\bigr)_i \text{ for } i \in \{1, \dotsc, n\}\Biggr\}, \\
\left((F^{\rm hyp}_r)^* y\right)_i & = -y_i' \text{ on } (-\tau_i, 0).
\end{align*}
The control operator $G_r^{\rm hyp}\in \mathcal L\left(\mathbb{C}^m, D\left((F_r^{\rm hyp})^*\right)'\right)$ is defined by the same formula \eqref{defG} as \(G_r\), replacing only \(D(F_r^*)'\) by \(D\left((F_r^{\rm hyp})^*\right)'\).

The proof of Theorem~\ref{th:iff_controllability} carries over to the present framework after replacing the spaces and operators by their hyperbolic counterparts. More precisely, the counterpart of Lemma~\ref{lem:key_approximate_eigenvector}
is obtained by setting
$w=H^{\mathrm{hyp}}(p)^*v$ and $z_i(\theta)=e^{-\bar p\theta}\bigl(v_i-\chi_i(\theta)w_i\bigr)$, 
where the fixed functions $\chi_i\in C^1([-\tau_i,0])$
satisfy $\chi_i(-\tau_i)=1$ and vanish near zero, so that
$z(0)=v$, $z_i(-\tau_i)=(K^*v)_i$, and the analogues of
\eqref{lepointcle}--\eqref{eq:lower-bound-z} follow by the same computations.
Moreover, setting $\tau_{\min}=\min_i\tau_i$, the condition
$\operatorname{rank}[K,B]=n$ gives, for $\Re p<0$
and $\|g\|_{\mathbb C^n}=1$,
\[
\|g^T H^{\mathrm{hyp}}(p)\|_{\mathbb C^n}
+\|g^TB\|_{\mathbb C^m}\ge
\|g^TK\|_{\mathbb C^n}
+\|g^TB\|_{\mathbb C^m}
-e^{\Re p\,\tau_{\min}}\ge \kappa-e^{\Re p\,\tau_{\min}},
\]
for some $\kappa>0$ independent of $p$ and $g$, yielding
the counterpart of \eqref{coer} for sufficiently negative
$\Re p$.

The proof of Theorem~\ref{main_result1} can be adapted in the same way. Indeed, the condition $\rank [K,B]=n$ allows one to choose a preliminary feedback such that $K$ is replaced by an invertible matrix $\widetilde K=K+BK_0$. For the corresponding homogeneous system,
$
(y_1(t-\tau_1)\cdots y_n(t-\tau_n))^T
=
\widetilde K^{-1}y(t)$,
which yields the group property by backward iteration. The remainder of
the proof is unchanged.
We summarize these results in the following theorem.

\begin{theorem}
\label{main_result1hypsysfirst}
Let \(r \in [1, +\infty)\). Then the hyperbolic system~\eqref{eq_hyp}--\eqref{eq_bord} is exactly controllable in $L^r((0, 1), \mathbb{C}^n)$ or, equivalently, the associated delay system~\eqref{eq_generic_delay} is exactly controllable in $\mathsf S^r$, if and only if the following conditions hold:
\begin{enumerate}[1)]
\item \label{assumption1} $\rank [M, B] = n$ for every $M \in \overline{H_{\rm hyp}(\mathbb{C})}$,
\item \label{assumption2} $\rank [K, B] = n$.
\end{enumerate}
Moreover, if these two rank conditions hold, then the system~\eqref{eq_hyp}--\eqref{eq_bord} is rapidly stabilizable in $L^2((0, 1), \mathbb{C}^n)$, or, equivalently, the associated delay system is rapidly stabilizable in $\mathsf S^2$.
\end{theorem}

\bibliographystyle{abbrv} 
\bibliography{bibliography}

\begin{thebibliography}{10}

\bibitem{auriol2019explicit}
J.~Auriol and F.~Di~Meglio.
\newblock An explicit mapping from linear first order hyperbolic {PDE}s to
  difference systems.
\newblock {\em Systems Control Lett.}, 123:144--150, 2019.

\bibitem{baratchart}
L.~Baratchart, S.~Fueyo, G.~Lebeau, and J.-B. Pomet.
\newblock Sufficient stability conditions for time-varying networks of
  telegrapher's equations or difference-delay equations.
\newblock {\em SIAM J. Math. Anal.}, 53(2):1831--1856, 2021.

\bibitem{bastin2016stability}
G.~Bastin and J.-M. Coron.
\newblock {\em Stability and boundary stabilization of 1-{D} hyperbolic
  systems}, volume~88 of {\em Progress in Nonlinear Differential Equations and
  their Applications}.
\newblock Birkh\"{a}user/Springer, [Cham], 2016.
\newblock Subseries in Control.

\bibitem{Brayton1966Bifurcation}
R.~K. Brayton.
\newblock Bifurcation of periodic solutions in a nonlinear
  difference-differential equations of neutral type.
\newblock {\em Quart. Appl. Math.}, 24(3):215--224, 1966.

\bibitem{chitour:hal-03827918}
Y.~Chitour, S.~Fueyo, G.~Mazanti, and M.~Sigalotti.
\newblock Hautus--{Y}amamoto criteria for approximate and exact controllability
  of linear difference delay equations.
\newblock {\em Discrete Contin. Dyn. Syst.}, 43(9):3306--3337, 2023.

\bibitem{CHITOUR2024_contr_hyp}
Y.~Chitour, S.~Fueyo, G.~Mazanti, and M.~Sigalotti.
\newblock Approximate and exact controllability criteria for linear
  one-dimensional hyperbolic systems.
\newblock {\em Systems Control Lett.}, 190:Paper No. 105834, 12, 2024.

\bibitem{Chitour2016Stability}
Y.~Chitour, G.~Mazanti, and M.~Sigalotti.
\newblock Stability of non-autonomous difference equations with applications to
  transport and wave propagation on networks.
\newblock {\em Netw. Heterog. Media}, 11(4):563--601, 2016.

\bibitem{Chitour2020Approximate}
Y.~Chitour, G.~Mazanti, and M.~Sigalotti.
\newblock Approximate and exact controllability of linear difference equations.
\newblock {\em J. \'{E}c. polytech. Math.}, 7:93--142, 2020.

\bibitem{Cooke1968Differential}
K.~L. Cooke and D.~W. Krumme.
\newblock Differential-difference equations and nonlinear initial-boundary
  value problems for linear hyperbolic partial differential equations.
\newblock {\em J. Math. Anal. Appl.}, 24:372--387, 1968.

\bibitem{coron}
J.-M. Coron.
\newblock {\em Control and nonlinearity}, volume 136 of {\em Mathematical
  Surveys and Monographs}.
\newblock American Mathematical Society, Providence, RI, 2007.

\bibitem{coron2017finite}
J.-M. Coron, L.~Hu, and G.~Olive.
\newblock Finite-time boundary stabilization of general linear hyperbolic
  balance laws via {F}redholm backstepping transformation.
\newblock {\em Automatica J. IFAC}, 84:95--100, 2017.

\bibitem{CoNg}
J.-M. Coron and H.-M. Nguyen.
\newblock Dissipative boundary conditions for nonlinear 1-{D} hyperbolic
  systems: sharp conditions through an approach via time-delay systems.
\newblock {\em SIAM J. Math. Anal.}, 47(3):2220--2240, 2015.

\bibitem{Coron2019Optimal}
J.-M. Coron and H.-M. Nguyen.
\newblock Optimal time for the controllability of linear hyperbolic systems in
  one-dimensional space.
\newblock {\em SIAM J. Control Optim.}, 57(2):1127--1156, 2019.

\bibitem{coron2020finite}
J.-M. Coron and H.-M. Nguyen.
\newblock Finite-time stabilization in optimal time of homogeneous quasilinear
  hyperbolic systems in one dimensional space.
\newblock {\em ESAIM Control Optim. Calc. Var.}, 26:Paper No. 119, 24, 2020.

\bibitem{Coron2021Null}
J.-M. Coron and H.-M. Nguyen.
\newblock Null-controllability of linear hyperbolic systems in one dimensional
  space.
\newblock {\em Systems Control Lett.}, 148:Paper No. 104851, 8, 2021.

\bibitem{CoronOptimal}
J.-M. Coron and H.-M. Nguyen.
\newblock On the optimal controllability time for linear hyperbolic systems
  with time-dependent coefficients.
\newblock {\em Ann. Inst. Fourier (Grenoble)}, 76(4):1635--1704, 2026.

\bibitem{Gantouh2020ApproximateCO}
Y.~El~Gantouh, S.~Hadd, and A.~Rhandi.
\newblock Approximate controllability of network systems.
\newblock {\em Evol. Equ. Control Theory}, 10(4):749--766, 2021.

\bibitem{Engel2008Vertex}
K.-J. Engel, M.~Kramar~Fijav\v{z}, R.~Nagel, and E.~Sikolya.
\newblock Vertex control of flows in networks.
\newblock {\em Netw. Heterog. Media}, 3(4):709--722, 2008.

\bibitem{Engel2001}
K.-J. Engel and R.~Nagel.
\newblock {\em One-parameter semigroups for linear evolution equations}, volume
  194 of {\em Graduate Texts in Mathematics}.
\newblock Springer-Verlag, New York, 2000.

\bibitem{fueyo:hal-04830757}
S.~Fueyo.
\newblock Hale--{S}ilkowski criterion for exact controllability of difference
  delay systems and 1-{D} hyperbolic {PDE}s.
\newblock {\em Internat. J. Control}, 99(5):1495--1507, 2026.

\bibitem{Corona_fueyo}
S.~Fueyo and Y.~Chitour.
\newblock A corona theorem for an algebra of {R}adon measures with an
  application to exact controllability for linear controlled delayed difference
  equations.
\newblock {\em C. R. Math. Acad. Sci. Paris}, 362:851--861, 2024.

\bibitem{Greenberg1984Effect}
J.~M. Greenberg and T.~T. Li.
\newblock The effect of boundary damping for the quasilinear wave equation.
\newblock {\em J. Differential Equations}, 52(1):66--75, 1984.

\bibitem{Hale1993Introduction}
J.~K. Hale and S.~M. Verduyn~Lunel.
\newblock {\em Introduction to functional-differential equations}, volume~99 of
  {\em Applied Mathematical Sciences}.
\newblock Springer-Verlag, New York, 1993.

\bibitem{hale2002stabilization}
J.~K. Hale and S.~M. Verduyn~Lunel.
\newblock Strong stabilization of neutral functional differential equations.
\newblock {\em IMA J. Math. Control Inform.}, 19(1-2):5--23, 2002.
\newblock Special issue on analysis and design of delay and propagation
  systems.

\bibitem{hautus1969controllability}
M.~L.~J. Hautus.
\newblock Controllability and observability conditions of linear autonomous
  systems.
\newblock {\em Indag. Math.}, 31:443--448, 1969.
\newblock Nederl. Akad. Wetensch. Proc. Ser. A 72.

\bibitem{Hu2024Minimal}
L.~Hu and G.~Olive.
\newblock The minimal control time for the exact controllability by internal
  controls of 1{D} linear hyperbolic balance laws.
\newblock {\em ESAIM Control Optim. Calc. Var.}, 30:Paper No. 82, 13, 2024.

\bibitem{Hu2025Method}
L.~Hu and G.~Olive.
\newblock A method to determine the minimal null control time of 1{D} linear
  hyperbolic balance laws.
\newblock {\em J. Differential Equations}, 440:Paper No. 113455, 26, 2025.

\bibitem{HuOliveSurvey}
L.~Hu and G.~Olive.
\newblock Control of {1D} first-order linear hyperbolic systems: a brief
  survey.
\newblock To appear in Chinese Annals of Mathematics, Series B, 2026.

\bibitem{JacobPartington2004}
B.~Jacob and J.~R. Partington.
\newblock Admissibility of control and observation operators for semigroups: a
  survey.
\newblock In {\em Current trends in operator theory and its applications},
  volume 149 of {\em Oper. Theory Adv. Appl.}, pages 199--221. Birkh\"auser,
  Basel, 2004.

\bibitem{Komornik}
V.~Komornik.
\newblock Rapid boundary stabilization of linear distributed systems.
\newblock {\em SIAM J. Control Optim.}, 35(5):1591--1613, 1997.

\bibitem{krstic2008backstepping}
M.~Krstic and A.~Smyshlyaev.
\newblock Backstepping boundary control for first-order hyperbolic {PDE}s and
  application to systems with actuator and sensor delays.
\newblock {\em Systems Control Lett.}, 57(9):750--758, 2008.

\bibitem{Mazanti2017Relative}
G.~Mazanti.
\newblock Relative controllability of linear difference equations.
\newblock {\em SIAM J. Control Optim.}, 55(5):3132--3153, 2017.

\bibitem{Pazy1983Semigroups}
A.~Pazy.
\newblock {\em Semigroups of linear operators and applications to partial
  differential equations}, volume~44 of {\em Applied Mathematical Sciences}.
\newblock Springer-Verlag, New York, 1983.

\bibitem{Russell1978Controllability}
D.~L. Russell.
\newblock Controllability and stabilizability theory for linear partial
  differential equations: recent progress and open questions.
\newblock {\em SIAM Rev.}, 20(4):639--739, 1978.

\bibitem{Russell_Weiss94}
D.~L. Russell and G.~Weiss.
\newblock A general necessary condition for exact observability.
\newblock {\em SIAM J. Control Optim.}, 32(1):1--23, 1994.

\bibitem{salamon1984control}
D.~Salamon.
\newblock {\em Control and observation of neutral systems}, volume~91 of {\em
  Research Notes in Mathematics}.
\newblock Pitman (Advanced Publishing Program), Boston, MA, 1984.

\bibitem{Slemrod1971Nonexistence}
M.~Slemrod.
\newblock Nonexistence of oscillations in a nonlinear distributed network.
\newblock {\em J. Math. Anal. Appl.}, 36:22--40, 1971.

\bibitem{Staffans2005}
O.~Staffans.
\newblock {\em Well-posed linear systems}, volume 103 of {\em Encyclopedia of
  Mathematics and its Applications}.
\newblock Cambridge University Press, Cambridge, 2005.

\bibitem{Suarez}
A.~Su{\'a}rez.
\newblock {\em Analysis and Design of Autonomous Microwave Circuits}.
\newblock John Wiley \& Sons, 2009.

\bibitem{Tucsnak_Weiss}
M.~Tucsnak and G.~Weiss.
\newblock {\em Observation and control for operator semigroups}.
\newblock Birkh\"auser Advanced Texts: Basler Lehrb\"ucher. Birkh\"auser
  Verlag, Basel, 2009.

\bibitem{Urquiza2005Rapid}
J.~M. Urquiza.
\newblock Rapid exponential feedback stabilization with unbounded control
  operators.
\newblock {\em SIAM J. Control Optim.}, 43(6):2233--2244, 2005.

\bibitem{vazquez2011backstepping}
R.~Vazquez, M.~Krstic, and J.-M. Coron.
\newblock Backstepping boundary stabilization and state estimation of a \(2
  \times 2\) linear hyperbolic system.
\newblock In {\em 2011 50th IEEE Conference on Decision and Control and
  European Control Conference}, pages 4937--4942, 2011.

\bibitem{Vest2013Rapid}
A.~Vest.
\newblock Rapid stabilization in a semigroup framework.
\newblock {\em SIAM J. Control Optim.}, 51(5):4169--4188, 2013.

\bibitem{Weck1982Remark}
N.~Weck.
\newblock A remark on controllability for symmetric hyperbolic systems in one
  space dimension.
\newblock {\em SIAM J. Control Optim.}, 20(1):1--8, 1982.

\bibitem{Weiss1989Control}
G.~Weiss.
\newblock Admissibility of unbounded control operators.
\newblock {\em SIAM J. Control Optim.}, 27(3):527--545, 1989.

\bibitem{Weiss1989Observation}
G.~Weiss.
\newblock Admissible observation operators for linear semigroups.
\newblock {\em Israel J. Math.}, 65(1):17--43, 1989.

\bibitem{yamamoto1981realization}
Y.~Yamamoto.
\newblock Realization theory of infinite-dimensional linear systems. {P}arts
  {I} and {II}.
\newblock {\em Math. Systems Theory}, 15:55--77 and 169--190, 1981/82.

\bibitem{YamamotoRealization}
Y.~Yamamoto.
\newblock Pseudo-rational input/output maps and their realizations: a
  fractional representation approach to infinite-dimensional systems.
\newblock {\em SIAM J. Control Optim.}, 26(6):1415--1430, 1988.

\bibitem{yamamoto1989reachability}
Y.~Yamamoto.
\newblock Reachability of a class of infinite-dimensional linear systems: an
  external approach with applications to general neutral systems.
\newblock {\em SIAM J. Control Optim.}, 27(1):217--234, 1989.

\bibitem{Yamamoto_Willems}
Y.~Yamamoto and J.~C. Willems.
\newblock Behavioral controllability and coprimeness for a class of
  infinite-dimensional systems.
\newblock In {\em 2008 47th IEEE Conference on Decision and Control}, pages
  1513--1518, 2008.

\end{thebibliography}

\end{document}